\documentclass[10pt,a4paper]{article}
\usepackage[T1]{fontenc}
\usepackage[ansinew]{inputenc}
\usepackage{amssymb}
\usepackage{amsmath}
\usepackage{graphicx}
\usepackage{amsthm}
\usepackage{cite}
\usepackage{color}
\usepackage[normalem]{ulem}

\newtheorem{theorem}{\textbf{Theorem}}[section]
\newtheorem{lemma}{\textbf{Lemma}}[section]
\newtheorem{proposition}{\textbf{Proposition}}[section]
\newtheorem{corollary}{\textbf{Corollary}}[section]
\newtheorem{remark}{\textbf{Remark}}[section]
\newtheorem{definition}{\textbf{Definition}}[section]
\numberwithin{equation}{section}

\allowdisplaybreaks[4]

\newcounter{marnote}

\def\br{\begin{remark}}
\def\er{\end{remark}}
\def\bp{\begin{proposition}}
\def\ep{\end{proposition}}
\def\bc{\begin{corollary}}
\def\ec{\end{corollary}}
\def\bd{\begin{definition}}
\def\ed{\end{definition}}

\def\non{\nonumber}

\newcommand{\RR}{\mathbb R}

\newcommand{\mcS}{\mathcal S}
\DeclareMathOperator{\tr}{tr}
\newcommand\defeq{\stackrel{\scriptscriptstyle \mathrm{def}}=}

\newcommand\p{\partial}

\newcommand{\R}{{\mathbb R}}

\newcommand{\F}{\mathcal F}
\newcommand{\E}{\mathcal E}
\newcommand{\Id}{\mathbb I}
\newcommand{\HH }{\mathcal H}

\newcommand{\bulk}{_\mathrm{bulk}}
\newcommand{\el}{_\mathrm{elastic}}
\newcommand{\tot}{_\mathrm{total}}
\newcommand{\T}{\mathbb{T}^2}
\newcommand{\bu}{\mathbf{u}}

\renewcommand{\le}{\leqslant}
\renewcommand{\geq}{\geqslant}
\renewcommand{\leq}{\leqslant}

\begin{document}

\title{Global Well-posedness of the Two Dimensional Beris--Edwards System with General Laudau--de Gennes Free Energy}

\author{
  Yuning Liu\thanks{NYU Shanghai, 1555 Century Avenue, Shanghai 200122, China.
  Email: \texttt{yl67@nyu.edu}.}
\and
 Hao Wu\thanks{
    School of Mathematical Sciences, Key Laboratory of Mathematics for Nonlinear Science (Fudan University), Ministry of
Education, Shanghai Key Laboratory for Contemporary Applied Mathematics, Fudan University, Shanghai 200433, China.
Email: \texttt{haowufd@fudan.edu.cn}.}
\and
  Xiang Xu\thanks{Department of Mathematics and Statistics,
  Old Dominion University, Norfolk, VA 23529, USA.
  Email: \texttt{x2xu@odu.edu}.}
}

\date{}
\maketitle
\begin{abstract}
In this paper, we consider the Beris--Edwards system for incompressible nematic liquid crystal flows.
The system under investigation consists of the Navier--Stokes equations for the fluid velocity $\bu$ coupled with an evolution equation for the order parameter $Q$-tensor.
One important feature of the system is that its elastic free energy takes a general form and in particular, it contains a cubic term that possibly makes it unbounded from below.
In the two dimensional periodic setting, we prove that if the initial $L^\infty$-norm of the $Q$-tensor is properly small,
then the system admits a unique global weak solution. The proof is based on the construction of a specific approximating system that preserves the $L^\infty$-norm of the $Q$-tensor along the time evolution.\medskip



\end{abstract}

\section{Introduction}
\setcounter{equation}{0}

The Landau--de Gennes theory is a fundamental continuum theory that
describes the state of nematic liquid crystals \cite{dG93,B12}. In
this framework, the local orientation and degree of ordering for the
liquid crystal molecules are modelled by a symmetric, traceless
$d\times d$ matrix $Q$ in $\mathbb{R}^d$ ($d=2,3$), known as the
$Q$-tensor order parameter \cite{BM10, BZ11}. The so-called
Landau--de Gennes free energy functional is   a nonlinear
integral functional of the $Q$-tensor and its spatial derivatives
\cite{B12}:
\begin{equation}\label{def of LG energy}
  \E(Q)=\int_{\Omega}\F (Q(x),\,\nabla Q(x))\,dx,
\end{equation}
where $Q$ is a matrix valued function defined on the spatial domain $\Omega\subset \R^d$ that
takes values in the configuration space
\begin{equation}
 \mathcal{S}^{(d)}_0\defeq \{M\in\RR^{d\times d}:\,M=M^{T},\,\tr(M)=0\}.\nonumber
\end{equation}
The energy density function $\F$ in \eqref{def of LG energy} is composed of
an elastic part and a bulk part \cite{BM10,MN14}:
\begin{equation}\label{energy-density}
   \F(Q,\nabla Q)\defeq \F\el(Q,\nabla Q)+\F\bulk(Q).
\end{equation}
The bulk free-energy describes the isotropic-nematic phase
transition. Its density function $\mathcal{F}\bulk$ is typically a truncated expansion in the scalar invariants of the $Q$-tensor,
which in the simplest setting takes the following form \cite{dG93}:
\begin{equation}\label{bulk-energy}
\mathcal{F}\bulk(Q)\defeq
\dfrac{a}{2}\mathrm{tr}(Q^2)-\dfrac{b}{3}\tr(Q^3)+\dfrac{c}{4}\tr^2(Q^2).
\end{equation}
Here, $a,\, b,\, c\in \mathbb{R}$ are material-dependent and temperature-dependent
constants. On the other hand, the elastic free energy characterizes
the distortion effect of the liquid crystal and its density function
$\mathcal{F}\el$ gives the strain energy density due to spatial
variations in the $Q$-tensor:
\begin{equation}\label{elastic-energy}
 \mathcal{F}\el(Q,\nabla Q) \defeq L_1\partial_kQ_{ij}\partial_kQ_{ij}+L_2\partial_jQ_{ik}\partial_kQ_{ij}
+L_3\partial_jQ_{ij}\partial_kQ_{ik}+L_4Q_{lk}\partial_kQ_{ij}\partial_lQ_{ij},
\end{equation}
for $1\leq i,\, j,\, k,\, l\leq d$. Throughout this paper, we use
the Einstein summation convention over repeated indices. The
coefficients $L_1,\, L_2,\, L_3,\, L_4$ are material-dependent
elastic constants. We note that $\mathcal{F}\el$ consists of three
independent terms associated with $L_1,\, L_2,\, L_3$ that are
quadratic in the first order partial derivatives of the $Q$-tensor, plus a
cubic term associated with the coefficient $L_4$. In the literature,
the case $L_2 = L_3 = L_4 = 0$ is usually called isotropic,
otherwise anisotropic if at least one of $L_2$, $L_3$, $L_4$ does
not vanish. In particular, the retention of the cubic term (i.e.,
$L_4\neq 0$) is due to the physically relevant consideration that it
allows a complete reduction of the Landau--de Gennes energy $\E(Q)$
to the classical Oseen--Frank energy for nematic liquid crystals \cite{M10,BZ11} (see also
\cite[Appendix B]{IXZ14}). There exists a vast recent literature on
the mathematical study of the Landau--de Gennes theory, and we refer
 interested readers to \cite{ADL14, ADL15, B12, BM10, BZ11, BE94,
CSX16, CRWX15, CX15, DMM17, DFRSS14, D15, DAZ, FRSZ14, FRSZ15, GR14,
GR15, HD14, HLWZZ15, HZ18, IXZ14, LW16, M10, MZ10, PZ11, PZ12,
WZZ15, W15, WXZ18} as well as the references cited therein.

In this paper, we study a basic model for the evolution of an incompressible nematic
liquid crystal flow, which was first proposed by Beris--Edwards \cite{BE94}.
The resulting PDE system consists of the Navier--Stokes equations
for the fluid velocity $\bu$ and nonlinear convection-diffusion equations
of parabolic type for the $Q$-tensor (see, e.g., \cite{ADL14,CRWX15,PZ11}).
To simplify the mathematical setting, throughout this paper we confine ourselves to
the two dimensional periodic case such that $d=2$ and $\Omega=\T$,
where $\mathbb{T}^{2}$ stands for the periodic box with period $\ell_i$ in the $i$-th direction with
$\mathcal{O}=(0,\ell_1)\times(0,\ell_2)$ being the periodic cell.
Without loss of generality, we simply set $\mathcal{O}=(0,1)^2$.
Then the coupled system we are going to study takes the following form:
\begin{align}
\p_t \bu +\bu\cdot\nabla{\bu}-\nu \Delta \bu+\nabla P&\,=\,\nabla\cdot(\sigma^a+\sigma^s),\qquad\qquad\, \forall\, (x,t) \in \T\times\mathbb{R}^+,\label{NSE}\\
\nabla\cdot{\bu}&\,=\,0,\qquad\qquad\qquad\qquad\quad\;\, \forall\, (x,t) \in \T\times\mathbb{R}^+,\label{incompressibility}\\
\p_t Q +\bu\cdot\nabla{Q}+Q\omega-\omega{Q}
&\,=\,\HH+\lambda\Id+\mu-\mu^{T},\qquad\;\;\, \forall\, (x,t) \in
\T\times\mathbb{R}^+.\label{Q equ}
\end{align}
The vector $\bu(x, t): \T \times \mathbb{R}^+ \rightarrow
\R^2$ denotes the velocity field of the fluid, $Q(x, t): \T \times
\mathbb{R}^+ \rightarrow \mathcal{S}_0^{(2)}$ stands for the matrix-valued
order parameter of liquid crystal molecules, and the scalar function
$P(x ,t):  \T \times\mathbb{R}^+\rightarrow\R$ is the hydrostatic
pressure. We assume that the system \eqref{NSE}--\eqref{Q equ} is
subject to the periodic boundary conditions
\begin{equation}\label{BC}
 \bu(x+e_i,t) = \bu(x,t), \quad Q(x+e_i,t)=Q(x,t), \quad \forall\, (x,t) \in \T\times\mathbb{R}^+,
\end{equation}
where $\{e_i\}_{i=1,2}$ is the canonical orthonormal basis of $\mathbb{R}^2$,
as well as the initial conditions
\begin{equation}\label{IC}
  \bu(x, 0)=\bu_0(x) \ \ \text{with}\ \nabla\cdot{\bu_0}=0, \quad Q(x, 0)=Q_0(x),\quad \forall\,x\in\T.
\end{equation}

The first equation \eqref{NSE} is the Navier--Stokes equation for $\bu$ with highly nonlinear
anisotropic force terms given by the stresses $\sigma^a$, $\sigma^s$,
while the second equation \eqref{incompressibility} simply gives the incompressibility condition. The third equation
\eqref{Q equ} describes the evolution of the $Q$-tensor, in which the left-hand side stands for the upper convective
derivative that represents the rate of change for a small particle of the material
that is rotating and stretching with the fluid flow.
We recall that when $d=2$ the general form of the upper convective derivative is given by
\begin{align}
&\partial_tQ+\bu\cdot\nabla{Q}-S(\nabla{\bu}, Q),\nonumber
\end{align}
where the tensor $S(\nabla{\bu}, Q)$ takes the following form
\begin{align}
&S(\nabla{\bu}, Q)\defeq (\xi A\bu+\omega)\big(Q+\frac12 \Id
\big)+\big(Q+\frac12 \Id \big)(\xi A\bu-\omega)-2\xi\big(Q+\frac12 \Id
\big)\tr(Q\nabla{\bu}),\label{S1}
\end{align}
with $\Id$ being the $2\times 2$ identity matrix and
\begin{align}
 A\bu=\dfrac{\nabla{\bu}+(\nabla{\bu})^{T}}{2},\quad \omega=\dfrac{\nabla{\bu}-(\nabla{\bu})^{T}}{2}\label{Do}
\end{align}
being the symmetric and skew-symmetric parts of the strain rate,
respectively. The parameter $\xi\in \R$ in \eqref{S1} depends on the
molecular details of a given liquid crystal material and it measures
the ratio between the tumbling and the aligning effects that a shear
flow may exert over the liquid crystal directors \cite{PZ11}. We
note that $\xi$ represents nontrivial interactions between the macro
fluid and the micro molecule configuration \cite{WXZ18}. In this
paper, we confine ourselves to the simple case with $\xi=0$ (cf.
\eqref{Q equ}), which is referred to as the co-rotational case in
the literature \cite{PZ12,W15}. On the right-hand side of equation
\eqref{Q equ}, $\HH$ represents the molecular field that is defined
to be minus the Fr\'echet derivative of the free energy
$\mathcal{E}(Q)$ with respect to $Q$ (without any constraint):
\begin{equation}\label{elder}
  \HH\defeq -\dfrac{\delta\E(Q)}{\delta{Q}},
\end{equation}
while $\lambda\in \R$ is a Lagrange multiplier corresponding to the
traceless constraint $\tr(Q)=0$ and $\mu\in\R^{2\times 2}$ is the Lagrange
multiplier corresponding to the constraint on symmetry $Q=Q^T$.
Through a straightforward calculation, we have (see Appendix)
\begin{align}\label{Hdef}
-\HH_{ij}&= -2L_1\Delta{Q_{ij}} -2(L_2+L_3) Q_{ik,kj}-2L_4 Q_{ij,l} Q_{l k,k} -2L_4 Q_{ij,l k}Q_{l k}\non\\
&\quad +L_4 Q_{kl,i} Q_{kl,j} +aQ_{ij}-bQ_{jk}Q_{ki}+c\tr(Q^2)Q_{ij}
\end{align}
as well as
\begin{equation}\label{eqnQExpansion}
\begin{split}
 &(\HH+\lambda\Id+\mu-\mu^{T})_{ij}\\
&\quad =\zeta\Delta{Q_{ij}} +2L_4 (  Q_{ij,l
}Q_{l k})_{,k} -L_4 Q_{kl,i} Q_{kl,j} +
\frac{L_4}{2}|\nabla{Q}|^2 \delta_{ij}-aQ_{ij}-c \tr(Q^2)Q_{ij},
\end{split}
\end{equation}
for $1\leq i,\,j,\, k,\, l \leq 2$, where we denote
\begin{equation}\label{def of zeta}
\zeta\defeq 2L_1+L_2+L_3.
\end{equation}
Returning to equation \eqref{NSE} for $\bu$,
the two highly nonlinear stress terms $\sigma^a$, $\sigma^s$ on its right-hand side are referred to as the anti-symmetric viscous stress and the distortion stress, respectively.
More precisely, they take the following form:
\begin{align}
\sigma^a&=Q(\HH+\lambda\Id+\mu-\mu^{T})-(\HH+\lambda\Id+\mu-\mu^{T}){Q},\label{tensora}\\
\sigma_{ij}^s&=-\frac{\partial\E(Q)}{\partial{Q}_{kl,j}}Q_{kl,i}\nonumber\\
   &=-2\big(L_1Q_{kl,i}Q_{kl,j}+L_2Q_{kj,l}Q_{kl,i}+L_3Q_{kl,l}Q_{kj,i}+L_4Q_{jm}Q_{kl,m}Q_{kl,i}\big),
\label{tensors}
\end{align}
for $1\leq i,\, j,\, k,\, l,\, m\leq 2$.

Throughout this paper, we assume for simplicity that $\nu$, the viscosity of the liquid crystal flow, is  a positive constant.
Moreover, we impose the following basic assumptions on the other coefficients:
\begin{align}
& L_4 \neq 0,\label{aL4}\\
& \kappa\defeq \min\{L_1+L_2, \ L_1+L_3\}>0,\label{coercivity}\\
& c>0. \label{structural-assumption}
\end{align}
We note that \eqref{coercivity} can be viewed as a sufficient condition for the coercivity of the system (see \cite[Appendix C]{IXZ14}),
while the condition \eqref{structural-assumption} ensures that the bulk part of the free
energy density $\mathcal{F}\bulk$ is bounded from below (see \cite{M10, MZ10}).
Moreover, from \eqref{def of zeta} and \eqref{coercivity} we easily infer the following relation
\begin{align}
\zeta\geq 2\kappa>0.\label{zk}
\end{align}
Under the current assumption \eqref{aL4}, the general system \eqref{NSE}--\eqref{Q equ} differs
from those that have been extensively studied in the literature, see for instance, \cite{ADL14, ADL15, CRWX15, DFRSS14,
D15, DAZ, FRSZ14, FRSZ15, GR14, GR15, HD14, LW16, PZ11, PZ12, WZZ15, W15}.
An important feature is that its free energy $\E(Q)$ now contains an unusual cubic term associated with the coefficient $L_4$,
which is physically meaningful but may cause the free energy functional $\E(Q)$ to be unbounded from below.
This fact will lead to essential difficulties in  the mathematical analysis of the system \eqref{NSE}--\eqref{Q equ}.

Our aim in this paper is to prove the existence and uniqueness of global weak solutions
to the Beris--Edwards system \eqref{NSE}--\eqref{IC} with general Laudau--de Gennes free energy under suitable assumptions.
To this end, we introduce  the definition of a weak solution.
\begin{definition}\label{def-weak-solution}
Suppose that $T\in (0, +\infty)$, $\bu_0\in L^2_\sigma(\T, \mathbb{R}^2)$
and $Q_0\in H^1(\T, \mathcal{S}_0^{(2)})\cap L^\infty(\T, \mathcal{S}_0^{(2)})$.
A pair $(\bu, Q)$ satisfying
\begin{align*}
&\bu\in H^1(0, T; (H^1_\sigma(\T))')\cap C([0, T]; L^2_\sigma(\T))\cap L^2(0, T;H_\sigma^1(\T)),\\
&Q\in H^1(0, T; L^{2}(\T))\cap C([0, T]; H^1(\T)) \cap L^2(0, T; H^2(\T)),\\
&Q\in L^\infty(0, T; L^\infty(\T)),\\
&\ \text{with}\ Q\in \mathcal{S}_0^{(2)}\quad\text{a.e. in } \ \T\times (0, T)
\end{align*}
is called a weak solution to the initial boundary problem \eqref{NSE}--\eqref{IC},
if
\begin{align}
& \int_0^T\langle \p_t \bu,\mathbf{v}\rangle_{(H^1)',H^1}\, dt
- \int_0^T\int_{\T}(\bu\otimes \bu):\nabla \mathbf{v}\,dxdt + \nu \int_0^T\int_{\T} \nabla{\bu}:
\nabla{\mathbf{v}}\,dxdt\nonumber\\
&\quad  = -\int_0^T\int_{\T} \left(\sigma^a(Q)+\sigma^s(Q)\right):\nabla \mathbf{v}\,dxdt,\label{eq:1.21}
\end{align}
 for every $\mathbf{v}\in L^2(0, T;H_\sigma^1(\T))$,
 \begin{align}\label{eq:1.19}
&\p_t Q_{ij}+u_k {Q}_{ij,k}+Q_{ik}\omega_{kj}-\omega_{ik}Q_{kj}\\
&\quad =\zeta\Delta{Q_{ij}} +2L_4 (Q_{ij,l}Q_{l k})_{,k} -L_4 Q_{k l,i} Q_{k l,j} +
\frac{L_4}{2}|\nabla{Q}|^2 \delta_{ij}-aQ_{ij}-c \tr(Q^2)Q_{ij}  ,
  \end{align}
a.e. in $\T\times (0, T)$, and the initial condition \eqref{IC} is satisfied. Moreover,
for every $t\in[0,T]$, it holds
\begin{equation}\label{BEL}
\E_{\tot}(t)+\nu\int_0^t\int_{\T}|\nabla{\bu}|^2\,dxds +
\int_0^t\int_{\T}\tr^2(\HH+\lambda\Id+\mu-\mu^{T})\,dxds=\E_{\tot}(0),
\end{equation}
where the total energy $\E_{\tot}$ is given by
\begin{equation}
\E_{\tot}(t)\defeq \frac12\int_{\T}|\bu(t,x)|^2dx+\E(Q(t)).\label{totE}
\end{equation}
\end{definition}
We are now in a position to state the main result of this paper.
\begin{theorem}\label{main-theorem}
Let $T>0$ be arbitrary. Under the assumptions \eqref{aL4}--\eqref{structural-assumption},
there exists a positive constant $\eta$ given by
\begin{align}
\eta= \min\left\{K_1\left(\dfrac{\kappa}{L_4}\right)^2,\ K_2\left(\frac{\zeta}{|L_4|}\right)\sqrt{\nu} \right\},
\label{bdeta}
\end{align}
where $K_1,\, K_2>0$ are two factors depending only on $\T$, such that if in addition,
the coefficients of system \eqref{NSE}--\eqref{Q equ} satisfy
\begin{align}
a\geq -c\eta,\label{coeaa}
\end{align}
then for any initial data $\bu_0\in L^2_\sigma(\T, \mathbb{R}^2)$,
$Q_0\in H^1(\T,\mathcal{S}_0^{(2)})\cap L^\infty(\T,\mathcal{S}_0^{(2)}) $ fulfilling
$$
\|Q_0\|_{L^\infty(\T)}< \sqrt{\eta},
$$
problem \eqref{NSE}--\eqref{IC} admits a unique global weak solution on $\T\times [0,T]$ in the sense of Definition \ref{def-weak-solution}.
 Furthermore, we have the uniform estimate
$$ \|Q(t)\|_{L^\infty(0,T; L^\infty(\T))}\leq \sqrt{\eta}.$$
\end{theorem}
\begin{remark}
(1) The restriction on the size of $\|Q_0\|_{L^\infty}$ is  due to the appearance of the cubic term associated with $L_4$.
As we can see from the definition of $\eta$, it is naturally determined by the ratios
$\kappa/|L_4|$ and $\zeta/|L_4|$, which measures the "good" part versus the "bad" part in the second order operator $\HH+\lambda\Id+\mu-\mu^{T}$.
Because the constants $K_1$, $K_2$ are from the application of the H\"older and Young inequalities as well as the elliptic estimates, they may depend on $\T$ at most.

(2) Although the technical assumption \eqref{coeaa} leads to a restriction on the coefficient $a$
(that in general depends linearly on the temperature of the material), it is still able to partially capture the physically
interesting regime of low temperature with $a\leq 0$ (i.e., the "the deep nematic" regime),
in which the isotropic state must lose stability, leaving only the uniaxial nematic states to be stable \cite{MZ10,MN14}.
\end{remark}

The Beris--Edwards system \eqref{NSE}--\eqref{Q equ} has been studied by many authors in recent
years. We recall some relevant results here. Concerning the
simplified system with $\xi=0$ and the elastic energy taking the
simplest form (i.e., $L_2=L_3=L_4=0$), in \cite{PZ12} the authors
obtained the existence of global weak solutions to the Cauchy
problem in $\mathbb{R}^d$ with $d=2,3$, and they proved results on
global regularity as well as weak-strong uniqueness for $d=2$. Some
improved results on the global well-posedness in two dimensions were
established in \cite{D15}, and we also refer to \cite{DMM17} for
certain regularity criterion in $\mathbb{R}^3$. Besides, long-time
behavior of global weak solutions to the Cauchy problem in
$\mathbb{R}^3$ was obtained in \cite{DFRSS14} by the method of
Fourier splitting. On the other hand, initial boundary value
problems subject to various boundary conditions have been
investigated by several authors, see for instance, \cite{ADL15,
GR14, GR15}, where the existence of global weak solutions, existence
and uniqueness of local strong solutions as well as some regularity
criteria were established. For the Beris--Edwards system with three
elastic constants $L_1, L_2, L_3$ but $L_4=\xi=0$, in \cite{HD14}
the authors studied the Cauchy problem in $\RR^3$ and proved the
existence of global weak solutions as well as the existence and
uniqueness of global strong solutions provided that the fluid
viscosity is sufficiently large. Next, concerning the full
Beris--Edward system with a general parameter $\xi\in \mathbb{R}$,
existence of global weak solutions for the Cauchy problem in
$\mathbb{R}^d$ with $d=2, 3$ was established in \cite{PZ11} for
sufficiently small $|\xi|$, while the uniqueness of weak solutions
for $d=2$ was given in \cite{DAZ}. Quite recently, global
well-posedness and long-time behavior of the system in a two
dimensional periodic setting were established in \cite{CRWX15},
without any smallness assumption on $|\xi|$. As far as the initial
boundary value problem subject to inhomogeneous mixed
Dirichlet/Neumann boundary conditions is concerned, in \cite{ADL14}
the authors treated the case with $\xi\neq 0$, $L_2=L_3=L_4=0$ and
proved the existence of global weak solutions as well as local
well-posedness with higher time-regularity for $d=2,3$. The local
well-posedness result was recently improved and extended to the case
with an anisotropic elastic energy in \cite{LW16}. At last, we would
like to mention that some modified versions of the Beris--Edwards
system in terms of its free energy have been investigated in the
literature. In \cite{W15}, the bulk potential \eqref{bulk-energy} is
replaced by a singular potential of Ball--Majumdar type (cf.
\cite{BM10}) that ensures the $Q$-tensor always stays in the
"physical" region. There, in the co-rotational regime $\xi=0$, the
author proved the existence of global weak solutions for $d=2, 3$ as
well as global strong solutions for $d=2$. Besides, some
non-isothermal variants of the Beris--Edwards system were recently
derived and analyzed in \cite{FRSZ14, FRSZ15}. The authors proved existence of
global weak solutions in the case of a singular potential under
periodic boundary conditions for a general parameter $\xi$ and $d=3$.

It is worth pointing out that all the results mentioned above are
obtained under the crucial assumption $L_4=0$. More related to our
problem, the authors in \cite{CSX16, HZ18, CX15, IXZ14} investigated
a gradient flow generated by the free energy \eqref{def of LG
energy} with $L_4\neq 0$. The resulting parabolic equation may be
considered as a fluid-free version of the Beris--Edward system
\eqref{NSE}--\eqref{Q equ} by setting $\bu=\mathbf{0}$ in \eqref{Q
equ}. Although the gradient flow admits a dissipative energy law,
due to the fact that the free energy $\E(Q)$ can be unbounded from
below, in general one cannot expect any useful information on a
priori estimates of its solution. On the other hand, from
\eqref{elastic-energy} we observe that the cubic term associated
with $L_4$ can be controlled by the other positive definite
quadratic terms provided that $\|Q\|_{L^\infty(\T)}$ is suitably
small. It was shown in \cite{IXZ14} that the smallness of
$\|Q\|_{L^\infty(\T)}$ can be preserved during the time evolution
provided that its initial value is small enough. As a consequence,
the usual $H^1$-level information provided by the energy dissipation
in this gradient flow can be effectively used, and the authors of
\cite{IXZ14} were able to construct a suitable approximation scheme
 to prove its global well-posedness. Based on the same idea, a stable numerical scheme was derived in \cite{CSX16},
 which also yields a different approach for the existence of global weak solutions for $d=2$ under the smallness assumption of
$\|Q_0\|_{L^\infty(\T)}$.

For our Beris--Edward system \eqref{NSE}--\eqref{Q equ} with $L_4\neq 0$, the
analysis turns out to be much more involved due to the complicated
interactions between the fluid motion and the molecular alignment.
Nevertheless, in the co-rotational case $\xi=0$, we are able to show that
the preservation of the smallness of $\|Q_0\|_{L^\infty(\T)}$ as in
\cite{CSX16, CX15, IXZ14} is still kept during the time evolution under
additional effects of advection and rotation due to the fluid (see
Lemma \ref{prop on maximum principle} below). Because of
the highly nonlinear coupling between the macroscopic fluid flow and
the microscopic orientational configuration of liquid crystal
molecules in our system \eqref{NSE}--\eqref{Q equ}, there are several extra
difficulties in the mathematical analysis, and none of those
approximate schemes utilized in \cite{CSX16, CX15, IXZ14} can be
applied to prove the existence of global weak solutions.
Alternatively, we introduce a regularized system (see
\eqref{NSEa}--\eqref{Q equa} below)
 where a higher-order regularizing term $\delta(-\Delta)^k$ with $\delta>0$,
$k\geq 4$ is added to the Navier--Stokes equation \eqref{NSE}.
It yields certain improved regularity of the velocity field $\bu$
and helps to guarantee the preservation of $L^\infty$-norm of the $Q$-tensor in this
regularized system. Assuming a slightly more regular initial datum that belongs to $L^2_\sigma(\T)\times H^2(\T)$,
local well-posedness of the regularized system can be proved by
a nonstandard application of Banach's fixed-point theorem, where
a suitable nonlinear mapping $\mathcal{Y}$ on the space $L^2(0,T;H^1(\T))$
is constructed in a delicate way (see \eqref{def of F} below).
In this process, the smallness of $\|Q_0\|_{L^\infty(\T)}$ plays an important role
in proving the contractivity of $\mathcal{Y}$.
Then combining the key property on preservation of $L^\infty$-norm of the $Q$-tensor together with the energy method,
we are able to derive uniform-in-time estimates that enable us to extend
the local solution to be a global one that is defined on an arbitrary time interval $[0,T]$.
Based on the dissipative energy law (see Proposition \ref{proposition on energy law}) and
using the compactness argument, we are able to prove the existence of global weak solutions
to our Beris--Edward system \eqref{NSE}--\eqref{Q equ} by passing to the limit as $\delta\rightarrow 0^+$.
Finally, global existence result for general initial data $(\bu_0, Q_0)\in L^2_\sigma(\T)\times (H^1(\T)\cap L^\infty(\T))$
can be achieved by employing a density argument.
Because of the highly nonlinear coupling terms related to the stress tensors $\sigma^a$ and $\sigma^s$,
uniqueness of global weak solutions to system \eqref{NSE}--\eqref{Q equ} is a nontrivial issue even in two spatial dimensions.
Some higher-order terms cannot be controlled if we perform energy estimates for the difference
of two solutions at the level of natural energy space, say $(\bu, Q)\in L^2_\sigma(\T)\times H^1(\T)$.
Inspired by \cite{LLT13, LTX16}, we choose to prove a continuous dependence result with respect to the
initial data in the lower-order energy space $(H^1_\sigma(\T))'\times L^2(\T)$
 (see Lemma \ref{prop on continuous dependence}), which immediately yields the uniqueness of global weak solutions.
Again in the proof, the preservation of $L^\infty$-norm of the
$Q$-tensor plays an essential role and furthermore, a special
cancellation relation hidden in the coupling structure has to be exploited.

The rest of this paper is organized as follows. In Section 2, we
state our notational conventions and then present some technical
lemmas. In Section 3, we derive the dissipative energy law satisfied
by the global weak solution to problem \eqref{NSE}--\eqref{IC} and a
weak maximum principle for the $Q$-equation, which yields the
preservation of $\|Q\|_{L^\infty}$ along time evolution. In Section 4, we introduce the
regularized problem and prove its global well-posedness with
slightly more regular initial data, i.e., $(\bu_0, Q_0)\in
L^2_\sigma(\T)\times H^2(\T)$. In the final Section 5, we complete
the proof of Theorem \ref{main-theorem} by passing to the limit as
$\delta\to 0^+$ in the approximate problem together with a density
argument for the initial data. Moreover, the continuous dependence
result is established. In the Appendix, we present some detailed
computations mentioned in the previous sections.

\section{Preliminaries}
\setcounter{equation}{0}

\subsection{Functional settings and notations}

Let $X$ be a real Banach space with norm $\|\cdot\|_X$ and $X'$ be its dual space.
By $\langle\cdot,\cdot\rangle_{X',X}$ we indicate the duality product between $X$ and $X'$.
We denote by $L^p(\T, M)$, $W^{m,p}(\T, M)$, $1\leq p\leq +\infty$, $m\in \mathbb{N}$,
the usual Lebesgue and Sobolev spaces defined on the torus $\T$ for $M$-valued functions
(e.g., $M=\mathbb{R}$, $M=\mathbb{R}^2$ or $M=\mathbb{R}^{2\times 2}$)
that are in $L^p_{loc}(\mathbb{R}^2)$ or $W^{m,p}_{loc}(\mathbb{R}^2)$
and periodic in $\T$, with norms denoted by $\|\cdot\|_{L^p}$, $\|\cdot\|_{W^{m,p}}$, respectively.
For $p=2$, we simply denote $H^m(\T)= W^{m,2}(\T)$ with norm $\|\cdot \|_{H^m}$. In particular for $m=0$,
we have $H^0(\T)=L^2(\T)$ and the inner product on $L^2(\T)$ will be denoted by $(\cdot, \cdot)_{L^2}$.
For simplicity, we shall not distinguish functional spaces when scalar-valued,
vector-valued or matrix-valued functions are involved if they are clear from the context.

For arbitrary vectors $\mathbf{u}, \mathbf{v}\in\R^2$,
we denote by $\mathbf{u}\cdot \mathbf{v}$ the scalar product in $\R^2$,
while for arbitrary matrices $A, B\in\mathbb{R}^{2\times 2}$,
the Frobenius product between $A$ and $B$ is defined by $A: B=\tr(A^T B)$.
For any matrix $Q\in \R^{2\times 2}$, its Frobenius norm is given by
$|Q| =\sqrt{\tr(Q^T Q)}=\sqrt{Q_{ij}Q_{ij}}$.
Concerning the norms for spatial derivatives of matrices, we
denote $|\nabla Q|^2(x)=\p_kQ_{ij}(x)\p_kQ_{ij}(x)$ and more generally for the multi-index $\alpha=(\alpha_1,\alpha_2)\in (\mathbb{N}\cup\{0\})^2$ with $|\alpha|=\alpha_1+\alpha_2=l\in \mathbb{N}$, we denote
$|\p^{\alpha} Q|^2(x)=\sum_{|\alpha|= l}\p^{\alpha}{Q}_{ij}(x)\p^{\alpha}{Q}_{ij}(x)$, where $\partial^\alpha=\frac{\partial^{|\alpha|}}{\partial_{1}^{\alpha_1}\partial_{2}^{\alpha_2}}$ for $\alpha\neq (0,0)$ and we agree that $\partial^{(0,0)}Q=Q$.
Then Sobolev spaces for $Q$-tensors will be defined in terms of the above norms. For instance,
\begin{align*}
&L^2(\T, \mathcal{S}_0^{(2)}) = \Big\{Q: \T\rightarrow \mathcal{S}_0^{(2)},\ \int_{\T}|Q(x)|^2\,dx < \infty \Big\},\\
&H^m(\T, \mathcal{S}_0^{(2)}) = \Big\{Q: \T\rightarrow \mathcal{S}_0^{(2)},\ \sum_{l=0}^m\sum_{|\alpha|=l}\int_{\T} |\partial^\alpha Q(x)|^2 \,dx < \infty \Big\},\quad m\in \mathbb{N}.
\end{align*}
 For any
$2\times 2$ differentiable matrix-valued function $Q=(Q_{ij})$, we denote the partial derivative of
its $ij$-component by $Q_{ij,k}=\p_k Q_{ij}$ and its divergence by $(\nabla \cdot Q)_i=\p_j Q_{ij}$, $1\leq i, j\leq 2$.

Next, we recall the well-established functional settings for periodic solutions
to the Navier--Stokes equations (see e.g., \cite{T95}):
\begin{align}
L^2_\sigma(\T)=\big\{\mathbf{v} \in L^2(\T, \R^2) ,\ \nabla\cdot
\mathbf{v}=0\big\},\quad H^m_\sigma =\big\{\mathbf{v}\in
H^m(\T,\R^2), \; \nabla\cdot \mathbf{v}=0\big\},
\end{align}
where $L^2_\sigma(\T)$ denotes the usual space of solenoidal vector field.
We note that in the spatial periodic setting, the Stokes operator is simply given by
 \begin{equation}
 \mathbf{S} \bu=-\Delta \bu, \quad  \forall\, \bu\in D(\mathbf{S})\defeq\{\bu\in L^2_\sigma(\T), \, \Delta \bu\in L^2_\sigma(\T)\}=H^2_\sigma(\T).\non
\end{equation}
The operator $\mathbf{S}$ can be seen as an unbounded positive linear self-adjoint operator on $L^2_\sigma(\T)$ and
it becomes an isomorphism from $\dot{H}^2_\sigma(\T)$ onto $\dot{L}^2_\sigma(\T)$, where dot means a function space with the constraint of zero mean \cite{T95}.

Throughout this paper, we denote by $C$ a generic constant that may depend on $\nu$, $L_i$, $a$, $b$, $c$, $\T$ and the
initial data $(\mathbf{u}_0, Q_0)$, whose value is allowed to vary on
occurrence. Specific dependence will be pointed out explicitly
if necessary.

\subsection{Useful lemmas}

Below we present some preliminary results that will be used in the subsequent
proofs. First of all, the following algebraic lemma turns to be
useful:
\begin{lemma}\label{cancel}
For any symmetric tensors $Q, M\in \R^{2\times2}$ and a vector field $\bu\in \R^2$, it holds
\begin{equation}\label{cancellation}
(QM-MQ): \nabla \bu =(Q\omega-\omega Q): M,
\end{equation}
where $\omega$ is given by \eqref{Do}.
\end{lemma}
\begin{proof}
Recall the definitions of $A\bu$ and $\omega$ given by \eqref{Do}.
Using the symmetry of $Q,\,M,\,A\bu$, the anti-symmetry of $\omega$ and property of the trace of matrix,
we infer from a direct computation that
 \begin{align*}
    (Q  M- M Q):\nabla \bu
    & =  (Q M- M Q):(A \bu+\omega)  = (Q M- M Q) : \omega\\
    & = Q_{ij}M_{jk}\omega_{ik}-M_{ij}Q_{jk}\omega_{ik} = -Q_{ij}M_{jk}\omega_{ki}
    +M_{ij}Q_{jk}\omega_{ki}\\
    & = -\tr (QM\omega)+\tr (MQ\omega)=-\tr (\omega QM)+\tr (Q\omega M)\\
    & =(Q\omega- \omega Q): M,
 \end{align*}
 which yields the required assertion.
\end{proof}

Next, we recall some standard elliptic estimates in $\T$:
\begin{lemma}\label{ellipes}
It holds that
\begin{align}
\label{H2}
\|Q\|_{H^2}\,&\leq\,C\left(\|\Delta Q\|_{L^2}\,+\,\|Q\|_{L^2}\right),\quad\ \  \forall\, Q\in H^2(\T),\\
\label{H3}
\|Q\|_{H^3}\,&\le\,C\left(\|\nabla\Delta Q\|_{L^2}\,+\,\|Q\|_{L^2}\right),\quad \forall\, Q\in H^3(\T),
\end{align}
 where the positive constant $C$ only depends on $\T$.
\end{lemma}

Besides, we present some interpolation inequalities.

\begin{lemma}
For any $f\in H^2(\T)$, the following inequalities are valid:
\begin{align}
& \|\nabla f\|_{L^4}^2\leq 3 \|f\|_{L^\infty}\|\Delta f\|_{L^2},\label{eq:1.27}\\
& \| \nabla f\|_{L^4}^2\leq C\| \nabla f\|_{L^2}\big(\|\Delta f\|_{L^2}+\|f\|_{L^2}\big). \label{eq:1.30}
\end{align}
Moreover, it holds
\begin{equation}\label{eq:1.31}
  \| \nabla f\|_{L^\infty}^2\leq C\| \nabla f\|_{L^2}\big(\|\nabla\Delta f\|_{L^2}+ \| f\|_{L^2}\big),\quad \forall\, f\in H^3(\T).
\end{equation}
In the above inequalities, the positive constant $C$ only depends on $\T$.
Besides, if $V\hookrightarrow H \hookrightarrow V'$ is a Gelfand-triple, then we have
\begin{align}
\|f\|_{C([0,T];H)}^2\leq 2\big(\|f\|_{H^1(0,T;V')}\|f\|_{L^2(0,T;V)}+\|f(\cdot, 0)\|_{H}^2\big).\label{Gel}
\end{align}
\end{lemma}
\begin{proof}
Concerning \eqref{eq:1.27}, we observe the identity
$$
|\p_j f|^4=\p_j(f\p_j f|\p_j f|^2)-3 f\p_j\p_j f|\p_j f|^2, \quad \forall\, f\in H^1(\T),
$$
and
\begin{align}
\int_{\T} |\partial_1\partial_2f|^2 dx=\int_{\T} (\partial_1\partial_1f)(\partial_2\partial_2f) dx,\quad \forall\, f\in H^2(\T).
\label{H2L2}
\end{align}
Then by H\"{o}lder's inequality, using integration by parts, we have  that
\begin{align*}
\|\nabla{f}\|_{L^4}^4&=\int_{\T}\sum_{j=1}^2|\partial_jf|^4dx=-3\int_{\T}\sum_{j=1}^2f\p_j\p_j
f|\p_j f|^2dx\\
&\leq
3\|f\|_{L^\infty}\|\nabla{f}\|_{L^4}^2\sqrt{\int_{\T}|\partial_1\partial_1f|^2+|\partial_2\partial_2f|^2dx}\\
&\leq
3\|f\|_{L^\infty}\|\nabla{f}\|_{L^4}^2\sqrt{\int_{\T}|\partial_1\partial_1f|^2+2|\partial_1\partial_2f|^2+|\partial_2\partial_2f|^2dx}\\
&=3\|f\|_{L^\infty}\|\nabla{f}\|_{L^4}^2\|\Delta{f}\|_{L^2},
\end{align*}
which yields the inequality \eqref{eq:1.27}. Inequalities
\eqref{eq:1.30} and \eqref{eq:1.31} follow from the classical Ladyshenskaya and Agmon
inequalities as well as the elliptic estimates in Lemma \ref{ellipes}. For the last inequality \eqref{Gel},
we refer to \cite[Section 2.1]{AB09} (see also \cite[Lemma 2.6]{ADL14}).
\end{proof}

We end this section by deriving some estimates for an elliptic problem related to \eqref{eqnQExpansion}.
\begin{lemma}\label{lemma1.0}
Consider the problem
\begin{align}
&\HH+\lambda\Id+\mu-\mu^{T}=g,\quad \forall\, x \in \T,\label{eq:1.13a}\\
&Q(x+e_i,t)=Q(x,t), \quad \forall\, x \in \T, \ \ i=1,2,\label{eq:1.13b}
\end{align}
where the left-hand side of \eqref{eq:1.13a} is given by \eqref{eqnQExpansion}.
Then for any constant $\eta$ satisfying
$$0<\eta< \frac{1}{121}\left(\frac{\zeta}{L_4}\right)^2,$$
if the solution $Q$ to problem \eqref{eq:1.13a}--\eqref{eq:1.13b} fulfills
$\|Q\|_{L^\infty(\T)}\leq \sqrt{\eta}$, then the following estimate holds
\begin{equation*}
\|Q\|_{H^2}\leq C(\|g\|_{L^2}+1)
\end{equation*}
     where the positive constant $C$ only depends on $a,\, c,\, L_4,\, \zeta,\, \eta$ and $\T$.
\end{lemma}
\begin{proof}
Taking the inner product of equation \eqref{eq:1.13a} with
$\Delta{Q}$  and integrating over $\T$, using the expression \eqref{eqnQExpansion},
H\"older's inequality as well as \eqref{eq:1.27}, we obtain that
\begin{align*}
\zeta\|\Delta{Q}\|_{L^2}^2&\leq
2|L_4|\int_{\T}\big|Q_{lk}Q_{ij,lk}\Delta{Q}_{ij}\big| dx+2|L_4|\int_{\T}\big|Q_{ij,l}Q_{lk,k}\Delta{Q}_{ij}\big|dx
\\
&\quad+|L_4|\int_{\T}\big|Q_{kl,i}Q_{kl,j}\Delta{Q}_{ij}\big|dx+\int_{\T}\Big|\big[aQ_{ij}+c\tr(Q^2)Q_{ij}\big]\Delta{Q}_{ij}\Big|dx\nonumber\\
&\quad +\int_{\T}|g_{ij}\Delta{Q}_{ij}|dx\\
&\leq
2|L_4|\|Q\|_{L^\infty}\|\Delta{Q}\|_{L^2}^2+3|L_4|\|\nabla{Q}\|_{L^4}^2\|\Delta{Q}\|_{L^2}\nonumber\\
&\quad +(|a|\|Q\|_{L^\infty}+c\|Q\|_{L^\infty}^3+\|g\|_{L^2})\|\Delta{Q}\|_{L^2}\nonumber\\
&\leq 11|L_4|\sqrt{\eta}\|\Delta{Q}\|^2 +(|a|\|Q\|_{L^\infty}+c\|Q\|_{L^\infty}^3+\|g\|_{L^2})\|\Delta{Q}\|_{L^2}.
\end{align*}
Using the assumption on $\eta$ and Young's inequality, we get
\begin{align*}
\zeta\|\Delta{Q}\|_{L^2}^2 & \leq 11|L_4|\sqrt{\eta}\|\Delta{Q}\|^2+ \frac{1}{2}\big(\zeta-11|L_4|\sqrt{\eta}\big)\|\Delta{Q}\|_{L^2}^2\non\\
&\quad +\frac12\big(\zeta-11|L_4|\sqrt{\eta}\big)^{-1}\big(|a|\eta^\frac12+c\eta^\frac{3}{2}+\|g\|_{L^2}\big)^2,
\end{align*}
which together with Lemma \ref{ellipes} easily yields the conclusion.
\end{proof}
\begin{remark}
Obviously, $\eta$ can be chosen arbitrary large as $|L_4|\to 0^+$. Indeed when $L_4=0$, we see from the above
proof that no smallness assumption on the $L^\infty$-norm of $Q$ is necessary in order to obtain the estimate for $\|Q\|_{H^2}$.
\end{remark}

\section{Basic Energy Law and Maximum Principle}
In this section, we first derive the following dissipative energy law for global weak solutions of problem \eqref{NSE}--\eqref{IC}.
\begin{proposition}\label{proposition on energy law}
 Let $(\bu,Q)$ be a weak solution of problem \eqref{NSE}--\eqref{IC} such that
\begin{align}
  &\bu\in H^1(0,T;(H^1_\sigma(\T))')\cap C([0,T]; L^2_\sigma(\T))\cap L^2(0,T;H^1_\sigma(\T)),\label{eq:1.26}\\
  & Q\in H^1(0,T;L^2(\T))\cap L^\infty(0,T;H^1(\T)\cap L^\infty(\T))\cap L^2(0,T;H^2(\T)).\label{eq:1.26a}
\end{align}
Then for any $t\in[0,T]$, we have
\begin{equation}\label{eq:1.25}
\E_{\tot}(t)+\nu\int_0^t\int_{\T}|\nabla{\bu}|^2\,dxds+
\int_0^t\int_{\T}\tr^2(\HH+\lambda\Id+\mu-\mu^{T})\,dxds =\E_{\tot}(0),
\end{equation}
where $\E_{\tot}$ is given by \eqref{totE}
\end{proposition}
\begin{proof}
The regularity properties \eqref{eq:1.26}--\eqref{eq:1.26a} guarantee that the solution $\bu$
and the quantity $\HH+\lambda\Id+\mu-\mu^{T} \in L^2(0,T;L^2(\T))$ can be used as test functions
for equations \eqref{NSE} and \eqref{Q equ}, respectively. Thus,
testing \eqref{NSE} by $\bu$ yields that
\begin{align}
  \frac 12 \frac d{dt}\int_{\T}|\bu|^2dx +\nu \int_{\T}|\nabla \bu|^2 dx
  +\int_{\T}\sigma^a : \nabla \bu dx=\int_{\T}\partial_j \sigma_{ij}^s u_i dx,\quad \text{a.e. in}\ (0,T),\label{Te1}
\end{align}
where one can verify the integrability of the right-hand side using \eqref{eq:1.26}--\eqref{eq:1.26a}
together with the Sobolev embedding theorem in two dimensions.
Keeping the symmetry property of $\HH+\lambda\Id+\mu-\mu^{T}$ in mind, we deduce from \eqref{cancellation}  that
\begin{align*}
   &\tr^2( \HH+\lambda\Id+\mu-\mu^{T})\\
   &\quad =(\HH+\lambda\Id+\mu-\mu^{T}):( Q_t+\bu\cdot\nabla{Q}+Q\omega-\omega{Q})\\
   &\quad = (\HH+\lambda\Id+\mu-\mu^{T}):(Q_t+ \bu\cdot\nabla Q) +( \HH+\lambda\Id+\mu-\mu^{T}):(Q\omega-\omega Q)\\
   &\quad = \HH :(Q_t+ \bu\cdot\nabla Q)  +\Big[Q( \HH+\lambda\Id+\mu-\mu^{T}) - ( \HH+\lambda\Id+\mu-\mu^{T})Q\Big]:\nabla \bu,
 \end{align*}
where we have used the identity
$(\lambda\Id+\mu-\mu^{T}):(\p_t Q+ \bu\cdot\nabla Q) =0$, since $\mu-\mu^{T}$ is skew-symmetric and $Q$ is traceless.
Then it follows that  for almost every $t\in (0,T)$,
\begin{align}
 &\int_{\T}\tr^2( \HH+\lambda\Id+\mu-\mu^{T})dx\nonumber\\
 &\quad =-\frac d{dt}\mathcal{E}(Q)+\int_{\T}\HH:(\bu\cdot\nabla Q)dx+\int_{\T}\sigma^a:\nabla \bu dx.
 \label{Te2}
\end{align}
Summing up the identities \eqref{Te1} and \eqref{Te2}, we obtain
\begin{align}
  &\frac d{dt}\left(  \frac 12 \int_{\T}
|\bu|^2 dx+\mathcal{E}(Q)\right)+\nu \int_{\T} |\nabla
\bu|^2dx +\int_{\T} \tr^2( \HH+\lambda\Id+\mu-\mu^{T})dx\nonumber\\
    &\quad =\int_{\T}\HH :(\bu\cdot\nabla Q) dx +\int_{\T}\partial_j \sigma_{ij}^su_i dx,\quad \text{a.e. in}\ (0,T), \label{vanish}
\end{align}
It remains to show that the right-hand side of \eqref{vanish} vanishes.
Since this step actually does not use the system \eqref{NSE}--\eqref{Q equ}, up to a suitable density argument,
we can simply assume that $(\bu,Q)$ is smooth enough. Moreover, by some straightforward computations, we have the following identities:
\begin{align*}
    u_i \partial_j(Q_{kl,i}Q_{kl,j})
    & =u_i Q_{kl,i} Q_{kl,jj}+u_i Q_{kl,ij}Q_{kl,j}\\
    & =u_i Q_{kl,i} \Delta Q_{kl}+\frac 12u_i \partial_i(Q_{kl,j}Q_{kl,j}),
\end{align*}
\begin{align*}
   u_i \partial_j(Q_{kj,l}Q_{kl,i})
   & =u_i Q_{kl,i} Q_{kj,lj}+u_i Q_{kj,l}Q_{kl,ij}\\
   & =u_i Q_{kl,i} Q_{kj,lj}+\frac 12u_i( Q_{kj,l}Q_{kl,ij}+  Q_{kl,j}Q_{kj,il})\\
   & =u_i Q_{kl,i} Q_{kj,lj}+\frac 12u_i\partial_i( Q_{kj,l}Q_{kl,j})\\
   & =u_i Q_{kj,i} Q_{kl,lj}+\frac 12u_i\partial_i( Q_{kj,l}Q_{kl,j})
\end{align*}
\begin{align*}
    u_i \partial_j(Q_{kl,l}Q_{kj,i})
    &=u_i Q_{kj,i} Q_{kl,lj}+u_i Q_{kl,l}Q_{kj,ij}\\
    &=u_i Q_{kj,i} Q_{kl,lj}+\frac 12 u_i (Q_{kl,l}Q_{kj,ij}+Q_{kj,j}Q_{kl,il})\\
    &=u_i Q_{kj,i} Q_{kl,lj}+\frac 12 u_i \partial_i (Q_{kl,l}Q_{kj,j})
\end{align*}
\begin{align*}
   &u_i\partial_j(Q_{jm}Q_{kl,m}Q_{kl,i})\\
   &\quad =u_iQ_{jm,j}Q_{kl,m}Q_{kl,i}+u_i Q_{jm}Q_{kl,mj}Q_{kl,i}+u_i Q_{jm}Q_{kl,m}Q_{kl,ij} \\
   &\quad =u_iQ_{jm,j}Q_{kl,m}Q_{kl,i}+u_i Q_{jm}Q_{kl,mj}Q_{kl,i}+\frac 12 u_i Q_{jm}\partial_i (Q_{kl,m}Q_{kl,j}) \\
   &\quad =u_iQ_{jm,j}Q_{kl,m}Q_{kl,i}+u_i Q_{jm}Q_{kl,mj}Q_{kl,i}\\
   &\qquad +\frac 12 u_i \partial_i(Q_{jm} (Q_{kl,m}Q_{kl,j}) )-\frac 12 u_i Q_{jm,i}  (Q_{kl,m}Q_{kl,j}).
\end{align*}
 Inserting the above four identities into \eqref{tensors} and using the divergence-free condition of $\bu$, we deduce that
 \begin{align*}
 &\int_{\T}   u_i\partial_j \sigma_{ij}^sdx\\
   &\quad=    -2\int_{\T} u_i\partial_j \big(L_1Q_{kl,i}Q_{kl,j}+L_2Q_{kj,l}Q_{kl,i}+L_3Q_{kl,l}Q_{kj,i}+L_4Q_{jm}Q_{kl,m}Q_{kl,i}\big)dx\\
   &\quad=    -2 \int_{\T}  \left(L_1u_i Q_{kl,i} \Delta Q_{kl}+L_2u_i Q_{kj,i} Q_{kl,lj}+L_3u_i Q_{kj,i} Q_{kl,lj}\right)dx\\
   &\qquad -2L_4\int_{\T}  \left(u_iQ_{jm,j}Q_{kl,m}Q_{kl,i}+u_i Q_{jm}Q_{kl,mj}Q_{kl,i}-\frac 12 u_i Q_{jm,i} Q_{kl,m}Q_{kl,j}\right)dx \\
   &\quad=    -2 \int_{\T} \left[ L_1u_i Q_{kl,i} \Delta Q_{kl}+(L_2+L_3)u_i Q_{kj,i} Q_{kl,lj}\right]dx\\
   &\qquad -2L_4\int_{\T} \left( u_iQ_{kl,i}Q_{jm,j}Q_{kl,m}+u_i Q_{kl,i}Q_{jm}Q_{kl,mj}\right)dx\\
   &\qquad +L_4\int_{\T}  u_i Q_{jm,i}   Q_{kl,m}Q_{kl,j}dx \\
   &\quad := -\int_{\T}  \widetilde{\HH}:(\bu\cdot\nabla Q) dx.
   \end{align*}
 It follows from \eqref{Hdef} that $ \widetilde{\HH} = \HH+aQ-bQ^2+c\tr(Q^2)Q$.
 On the other hand, by the incompressibility condition \eqref{incompressibility}, we see that
 \begin{align*}
 \int_{\T} \big[aQ-bQ^2+c\tr(Q^2)Q\big]:(\bu\cdot\nabla Q) dx=0.
 \end{align*}
As a consequence, it holds
 $$\int_{\T}   u_i\partial_j \sigma_{ij}^sdx=-\int_{\T}  \HH:(\bu\cdot\nabla Q) dx,$$
which implies that the right-hand side of \eqref{vanish} will simply vanish.
Then integrating \eqref{vanish} with respect to time, we arrive at our conclusion \eqref{eq:1.25}.
\end{proof}
The basic energy law \eqref{eq:1.25} will play an essential role in the analysis of problem \eqref{NSE}--\eqref{IC}.
However, we note that when $L_4\neq 0$, \eqref{eq:1.25}
fails to provide certain good a priori estimates for the weak solution $(\bu,
Q)$, since the free energy $\E(Q)$ may be unbounded from below. This
is very different from the case that has been considered in the
previous literature, where $L_4=0$ is always assumed. On the other
hand, the above mentioned difficulty can be overcome if an estimate on
$\|Q\|_{L^\infty_tL^\infty_x}$ is available. To this end, we consider the
following evolution equation for $Q$ with advection and rotation effects due to the fluid:
\begin{align}
&\p_t Q +\bu\cdot\nabla{Q}+Q\omega-\omega{Q}\,=\,\HH+\lambda\Id+\mu-\mu^{T},\quad \forall\, (x,t) \in \T\times\mathbb{R}^+,\label{Q equd}\\
&Q(x+e_i,t)=Q(x,t), \qquad\qquad\qquad\qquad\qquad\qquad \forall\, (x,t) \in \T\times\mathbb{R}^+,\label{BCd}\\
&Q(x, 0)=Q_0(x),\qquad\qquad\qquad\qquad\qquad\qquad\qquad\,\,
\forall\,x\in\T,\label{ICd}
\end{align}
where the right-hand side of \eqref{Q equd} is given by \eqref{eqnQExpansion}.
One interesting feature of the system \eqref{Q equd}--\eqref{ICd} is that the $L^\infty$-norm of its solution $Q$
will be preserved during the time evolution, provided that it is suitably small at $t=0$
(see \cite{IXZ14} for the fluid free case, which is a gradient flow generated by the free energy $\E(Q)$).
\begin{lemma}\label{prop on maximum principle}
Assume that $c>0$, $\zeta>0$ and $\bu\in C([0,T];L^2_\sigma(\T))\cap L^2(0,T;H^{k}_\sigma(\T))$ for some integer $k\geq 2$.
Let $Q\in H^1(0,T;L^2(\T))\cap C([0,T];H^1(\T)\cap L^\infty(\T))\cap L^2(0,T;H^2(\T))$ be a solution to problem \eqref{Q equd}--\eqref{ICd}.
For any constant $\eta$ satisfying $0<\eta < \frac{1}{9}\left(\frac{\zeta}{L_4}\right)^2$, if
\begin{equation}\label{small coefficient assumption}
 \|Q_0\|_{L^\infty}\leq \sqrt{\eta},
\end{equation}
and
\begin{align}
a \geq - c \eta,\label{coea}
\end{align}
then it holds
\begin{equation}\label{eq:1.20}
 \|Q\|_{L^\infty(0, T;L^\infty(\T))} \leq \sqrt{\eta}.
\end{equation}
\end{lemma}
\begin{proof}
We take the inner product of equation \eqref{Q equd} with the test
function $2(|Q|^2-\eta)^+Q$, where $(\cdot)^+$ denotes the non-negative part of a function, and then integrate over $\T$. By the
incompressibility condition $\nabla \cdot \bu=0$, formula \eqref{eqnQExpansion},
and the facts that $\omega=-\omega^T$, $Q\in\mcS^{(2)}_0$, we deduce, after integration by parts that
\begin{align*}
&\frac12\frac{d}{dt}\int_{\T}\big|(|Q|^2-\eta)^+\big|^2\,dx +\zeta\int_{\T}\big|\nabla(|Q|^2-\eta)^+\big|^2\,dx+2\zeta\int_{\T}|\nabla{Q}|^2(|Q|^2-\eta)^+\,dx\\
&\quad=
-\underbrace{6L_4\int_{\T}Q_{ij}\partial_iQ_{kl}\partial_jQ_{kl}(|Q|^2-\eta)^+\,dx}_{I_1}
-\underbrace{2L_4\int_{\T}Q_{ij}\partial_i(|Q|^2-\eta)^+\partial_j(|Q|^2-\eta)^+\,dx}_{I_2}\\
&\qquad
- 2\int_{\T}|Q|^2\big(a+c|Q|^2\big)(|Q|^2-\eta)^+\,dx.
\end{align*}
The first two terms on the right-hand side can be estimated as follows:
\begin{align*}
  |I_1|&\leq  6|L_4|\int_{\T}|Q||\nabla{Q}|^2(|Q|^2-\eta)^+\,dx, \\
  |I_2|&\leq  2|L_4|\int_{\T}|Q|\big|\nabla(|Q|^2-\eta)^+\big|^2\,dx.
\end{align*}
As a consequence, we obtain
\begin{align}\label{integral inequality}
\frac12 &\frac{d}{dt}\int_{\T}\big|(|Q|^2-\eta)^+\big|^2\,dx \non\\
&\leq
2\int_{\T}\big(3|L_4||Q|-\zeta\big)|\nabla{Q}|^2(|Q|^2-\eta)^+\,dx
+\int_{\T}\big(2|L_4||Q|-\zeta\big)\big|\nabla(|Q|^2-\eta)^+\big|^2\,dx\non\\
&\quad - 2\int_{\T}|Q|^2\big(a+c|Q|^2\big)(|Q|^2-\eta)^+\,dx.
\end{align}
To finish the proof, we argue by contradiction. Denote
$$
  f(t)=\|Q(t)\|_{L^\infty},
$$
which is a continuous function on $[0,T]$. According to the assumption
$f(0)\leq \sqrt{\eta}$, if $f(t)\leq \sqrt{\eta}$ for any $t\in [0, T]$,
then the conclusion automatic follows. Otherwise, there exists some $t_0\in (0,
T]$ and $s_0\in [0,t_0)$ such that
$$\sqrt{\eta}<f(t)\leq\frac{\zeta}{3|L_4|}\quad \text{for any}\ t \in (s_0, t_0],\quad \text{and}\ \ f(s_0)=\sqrt{\eta}.$$
From \eqref{coea}, we have
\begin{align}
&\int_{\T}|Q|^2\big(a+c|Q|^2\big)(|Q|^2-\eta)^+\,dx\nonumber\\
&\quad \geq c\int_{\T}|Q|^2(|Q|^2-\eta)(|Q|^2-\eta)^+\,dx\nonumber\\
&\quad \geq c\int_{\T}|Q|^2\left[(|Q|^2-\eta)^+\right]^2\,dx\geq 0,\quad \forall\, t \in [s_0, t_0].\nonumber
\end{align}
Then we infer from \eqref{integral inequality} that
$$
  \frac{d}{dt}\int_{\T}\big|(|Q|^2-\eta)^+\big|^2\,dx\leq 0, \quad
  \forall \,t \in [s_0, t_0),
$$
which together with the fact $f(s_0)=\sqrt{\eta}$ yields that
$$
  \int_{\T}\big|(|Q|^2-\eta)^+\big|^2(t)\,dx\leq 0, \qquad
  \forall\, t \in [s_0, t_0).
$$
Hence, $f(t)\leq \sqrt{\eta}$ for any $t \in [s_0, t_0)$, which leads to a
contradiction with the assumption $f(t_0)>\sqrt{\eta}$ and the continuity of $f(t)$.
\end{proof}

\section{Global Well-posedness of an Approximate System}
\setcounter{equation}{0}

In this section, we study an approximate system for our original problem \eqref{NSE}--\eqref{IC}, in which a higher-order dissipative term is added to the Navier--Stokes equation
\eqref{NSE}, for the sake of improving the regularity as well as integrability of the velocity field $\bu$.
Denote
$$
 \mathcal{L}_\delta:=\delta(-\Delta)^k,
$$
for some given positive constant $\delta$ and even integer $k\geq 4$. We consider the following regularized system
\begin{align}
\p_t \bu+\bu\cdot\nabla{\bu}+ \mathcal{L}_\delta \bu-\nu
\Delta{\bu}+\nabla P &=\nabla\cdot(\sigma^a+\sigma^s),\qquad\;\; \forall\,
(x,t) \in \T\times\mathbb{R}^+,
\label{NSEa}\\
\nabla\cdot{\bu}&=0,\qquad\qquad\qquad\quad\;\;\, \forall\, (x,t) \in \T\times\mathbb{R}^+,\label{incoma}\\
\p_t Q +\bu\cdot\nabla{Q}+Q\omega-\omega{Q} &=\HH
+\lambda\Id+\mu-\mu^{T},\quad \forall\, (x,t) \in
\T\times\mathbb{R}^+,\label{Q equa}
\end{align}
subject to the periodic boundary conditions and initial conditions
\begin{align}
 &\bu(x+e_i,t) = \bu(x,t), \quad Q(x+e_i,t)=Q(x,t), \qquad\qquad\; \forall\, (x,t) \in \T\times\mathbb{R}^+,
 \label{BCa}\\
&\bu(x, 0)=\bu_0(x) \ \ \text{with}\ \nabla\cdot{\bu_0}=0, \quad
Q(x, 0)=Q_0(x),\quad\ \forall\,x\in\T.\label{ICa}
\end{align}

The main result of this section is as follows:

\begin{proposition}\label{approx1}
Let $T>0$ be arbitrary. Under the assumptions \eqref{aL4}--\eqref{structural-assumption}, there exists a positive constant
\begin{align}
 \eta_1=C_1\left(\frac{\kappa}{L_4}\right)^2,
  \label{eta1}
\end{align}
where $C_1>0$ only depends on $\T$ such that if in addition, the coefficients $a$, $c$ fulfill $a\geq -c\eta_1$,
then for any initial data $(\bu_0, Q_0)\in L^2_\sigma(\T,\mathbb{R}^2)\times H^2(\T,\mathcal{S}_0^{(2)})$ with
\begin{equation}\label{eq:1.29}
 \|Q_0\|_{L^\infty}\leq \sqrt{\eta_1},
\end{equation}
the approximate system \eqref{NSEa}--\eqref{ICa} admits a unique global solution such that
 \begin{align}
  & \bu\in H^1(0, T; (H^{k}_\sigma(\T))')\cap C([0,T], L_\sigma^2(\T))\cap L^2(0, T; H^k_{\sigma}(\T)),\label{rregu}\\
  & Q\in H^1(0,T;H^1(\T))\cap C([0,T]; H^2(\T))\cap L^2(0,T;H^3(\T)).\label{rregQ}
 \end{align}
 Moreover, we have $Q\in C([0,T]; L^\infty(\T))$ with the uniform bound
 \begin{equation}\label{eq:1.20a}
 \|Q\|_{C([0, T];L^\infty(\T))} \leq \sqrt{\eta_1},
\end{equation}
  as well as the following dissipative energy law
\begin{align}
  &\frac 12\int_{\T} |\bu(\cdot, t)|^2 dx +\mathcal{E}(Q(t))+\int_0^t \int_{\T} \left(\delta|\nabla^k \bu(s)|^2+\nu |\nabla \bu(s)|^2 + \tr^2(  \HH+\lambda\Id+\mu-\mu^{T})\right)dxds\non\\
  &\quad = \frac 12\int_{\T} |\bu_0|^2dx +\mathcal{E}(Q_0), \quad\qquad \forall\, t\in [0,T].\label{eq:1.01}
\end{align}
\end{proposition}

The proof of Proposition \ref{approx1} consists of several steps.
Roughly speaking, we first prove the existence and uniqueness of a local-in-time
solution via a suitable fixed point argument. Then using the uniform estimates provided
by the dissipative energy law \eqref{eq:1.01}, we are able to extend the local solution to a global one.

\subsection{Local well-posedness}

We start with the local well-posedness of problem \eqref{NSEa}--\eqref{ICa}.

\medskip

\noindent \textbf{Step 1}. Reformulation of the system.

\medskip

Recalling the expression \eqref{eqnQExpansion}, we reformulate the system \eqref{NSEa}--\eqref{Q equa} into the following form:
\begin{align}
 \p_t \bu + \bu\cdot\nabla{\bu} + \mathcal{L}_\delta \bu-\nu \Delta{\bu}+\nabla P & = \nabla\cdot(\sigma^a+\sigma^s),\label{NSEb}\\
 \nabla\cdot{\bu} &=0,\\
 \p_t Q_{ij}-\zeta \Delta Q_{ij}-2L_4 (Q_{ij,l}(Q_0)_{lk})_{,k}
& =2L_4 (Q_{ij,l}Q_{l k})_{,k}-2L_4 (Q_{ij,l}(Q_0)_{l k})_{,k}+\mathcal{G}_{ij}(\bu,Q),
\label{Q equb}
\end{align}
  where
  \begin{align}
   \mathcal{G}_{ij}(\bu,Q)&:= -\bu\cdot\nabla{Q_{ij}}+\omega_{ik} Q_{kj} -Q_{ik}\omega_{kj} -aQ_{ij}-c
\tr(Q^2)Q_{ij} \non\\
&\quad\ -L_4 Q_{kl,i}Q_{kl,j} + \frac{L_4}{2}|\nabla{Q}|^2\delta_{ij}.
\label{eq:1.02}
  \end{align}
\begin{remark}
One can see that an extra term $-2L_4 (  Q_{ij,l }(Q_0)_{l k})_{,k}$ is
added to both sides of the equation  \eqref{Q equb}. Since $Q_0\in
H^2(\T)\hookrightarrow L^\infty(\T)$ and $\|Q_0\|_{L^\infty(\T)}$ is assumed to be relatively small (see \eqref{eq:1.29}), then the left-hand side of
 \eqref{Q equb} simply behaves like a linear heat equation.
 On the other hand, an advantage of adding the extra term $-2L_4 (Q_{ij,l }(Q_0)_{l k})_{,k}$ on the right-hand side of \eqref{Q equb}
 is that together with the nonlinear term $2L_4 (Q_{ij,l}Q_{l k})_{,k}$, it allows us to apply an interpolation argument
 that leads to the construction of a contraction mapping.
\end{remark}

\medskip
\noindent\textbf{Step 2}. Well-posedness of auxiliary problems for $\bu$ and $Q$.

\medskip

In order to deal with the highly nonlinear system \eqref{NSEb}--\eqref{Q equb},
we first present a preliminary result on the solvability of a higher-order Navier--Stokes type system.
\begin{lemma}\label{lemma1}
Let $\delta>0$, $k\geq 4$ be fixed.
For any $T>0$, consider the following problem
  \begin{align}
   & \p_t \bu +\bu\cdot\nabla{\bu}+ \mathcal{L}_\delta \bu-\nu\Delta{\bu}+\nabla P=g, \quad \forall\, (x,t) \in \T\times(0,T),\label{eq:1.12a}\\
   & \nabla \cdot \bu=0, \qquad\qquad\qquad\qquad\qquad\qquad\quad\, \forall\, (x,t) \in \T\times(0,T),\label{eq:1.12b}\\
   & \bu(x+e_i,t) = \bu(x,t), \qquad\qquad\qquad\qquad\;\, \forall\, (x,t) \in \T\times(0,T),\label{eq:1.12c}\\
   & \bu|_{t=0}=\bu_0(x), \qquad\qquad\qquad\qquad\qquad\quad\ \forall\, x \in \T.\label{eq:1.12d}
  \end{align}
  Then for any $\bu_0\in L^2_\sigma(\T)$ and $g\in L^2(0,T; (H^{k}_\sigma(\T))')$, problem \eqref{eq:1.12a}--\eqref{eq:1.12d} admits a unique global weak solution such that
  \begin{align}
  \bu\in H^1(0,T; (H^k_\sigma(\T))')\cap C([0,T]; L^2_\sigma(\T))\cap L^2(0,T; H^k_\sigma(\T)),\label{regup}
  \end{align}
  which satisfies
\begin{align}\label{eq:1.11}
 & \|\bu\|_{H^1(0,T;(H^k(\T))')\cap C([0,T];L^2(\T))\cap L^2(0,T;H^k(\T))}
  \leq C\left(\|g\|_{L^2(0,T;(H^{k}(\T))')}+\|\bu_0\|_{L^2}\right).
\end{align}
Let $\bu^i$ ($i=1,2$) be two weak solutions of problem \eqref{eq:1.12a}--\eqref{eq:1.12d} with external forces $g^i\in L^2(0,T;L^2(\T))$ and initial data $\bu^i(x,0)=\bu^i_0(x)\in L^2_\sigma(\T)$. Then we have
   \begin{equation}\label{eq:1.16}
    \| \bu^1- \bu^2\|_{C([0,T];L^2(\T))\cap L^2(0,T;H^k(\T))}\leq C\left(\|g^1-g^2\|_{L^2(0,T;(H^k(\T))')}+\|\bu^1_0-\bu^2_0\|_{L^2(\T)}\right),
  \end{equation}
  where the constant $C$ may depend on $\|\bu^i\|_{L^2(0,T;L^\infty(\T))}$, $T$ and $\T$.
Furthermore, if $g\in L^2(0,T;L^2(\T))$ and $\bu_0\in H^k_\sigma(\T)$, then the solution is more regular and
  \begin{align}\label{eq:1.14}
   & \|\bu\|_{H^1(0,T;L^2(\T))\cap C([0,T];H^k(\T))\cap L^2(0,T;H^{2k}(\T))}
    \leq C\left(\|g\|_{L^2(0,T;L^2(\T))}+\|\bu_0\|_{H^k(\T)}\right).
  \end{align}
All the constants above may depend on the parameter $\delta>0$.
\end{lemma}
\begin{proof}
The regularized system \eqref{eq:1.12a}--\eqref{eq:1.12d} was used for the well-posedness of the Navier--Stokes equations (see e.g., \cite{L69}).
Recalling that $k$ is even, then based on the coerciveness of
the bilinear form $a(\cdot,\cdot): H^k_\sigma(\T)\times H^k_\sigma(\T)\mapsto
\mathbb{R}$ (see e.g., \cite{brezis2010functional})
\begin{equation*}
  a(\mathbf{u},\mathbf{v})=\int_{\T}\Big( \delta \mathcal{D}^k \bu \cdot \mathcal{D}^k \mathbf{v} +\bu\cdot \mathbf{v}\Big)\,dx, \quad \text{with}\ \ \mathcal{D}^k=(-\Delta)^\frac{k}{2},
\end{equation*}
the proof can be carried out in a similar manner as in \cite[Proposition 4.1]{LT14}. Thus we omit the details here.
\end{proof}

Next, we consider a linear parabolic problem associated with \eqref{Q equb}.
\begin{lemma}\label{lemma-energy-estimate-1}
Suppose that the assumptions in Proposition \ref{approx1} are satisfied.
For any $T>0$ and $G\in L^2(0,T;H^1(\T,\mathcal{S}_0^{(2)}))$, the following linear problem
\begin{align}
&\partial_t Q_{ij}-\zeta \Delta Q_{ij}-2L_4 (  Q_{ij,l }(Q_0)_{l k})_{,k}=G_{ij},\quad\, \forall\, (x,t) \in \T\times(0,T),\label{Q equp}\\
&Q(x+e_i,t)=Q(x,t), \quad\qquad\qquad\qquad\qquad\;\;\ \,\forall\, (x,t) \in \T\times(0,T),\label{BCp}\\
&Q(x,0)=Q_0(x),\quad\qquad\qquad\qquad\qquad\qquad\;\;\ \
\forall\,x\in\T,\label{ICp}
\end{align}
admits a unique strong solution such that
\begin{align}
&Q\in H^1(0,T; H^1(\T))\cap C([0,T]; H^2(\T))\cap L^2(0,T; H^3(\T)),\label{regQp}\\
&Q\in \mathcal{S}_0^{(2)}\quad\text{a.e. in } \ \T\times (0, T)\nonumber
\end{align}
and
\begin{equation}\label{eq:1.15}
  \|Q\|_{H^1(0,T;H^{1}(\T))\cap  C([0,T]; H^2(\T)) \cap L^2(0,T;H^{3}(\T))} \leq C(\|Q_0\|_{{H^{2}}}+\|G\|_{L^2(0,T;H^1(\T))}).
\end{equation}
\end{lemma}
\begin{proof}
The assumptions \eqref{coercivity}, \eqref{eq:1.29} and the relation \eqref{zk} guarantee the ellipticity of the second order operator in equation \eqref{Q equp}
when $C_1$ is small enough. Thus
we can prove the existence of a unique strong solution that satisfies \eqref{regQp} by a standard Galerkin approximation.
Below we only show the validity of necessary a priori estimates.
Testing  \eqref{Q equp} with $Q-\Delta{Q}$ and integrating over $\T$, using integration by parts, \eqref{eq:1.27}--\eqref{eq:1.30}, H\"older's inequality and Young's inequality,
we see that
\begin{align*}
&\frac12\frac{d}{dt}\big(\|Q\|_{L^2}^2+\|\nabla{Q}\|_{L^2}^2)+\zeta\big(\|\nabla{Q}\|_{L^2}^2+\|\Delta{Q}\|_{L^2}^2)\\
&\quad =-2L_4\int_{\T}(Q_0)_{lk}Q_{ij,l}Q_{ij,k}dx -2L_4\int_{\T}(Q_0)_{lk}Q_{ij,lk}\Delta{Q}_{ij}dx\non\\
&\qquad -2L_4\int_{\T}(Q_0)_{lk,k}Q_{ij,l}\Delta{Q}_{ij}dx +\int_{\T}G_{ij}(Q_{ij}-\Delta{Q}_{ij}) dx\\
&\quad \leq 2|L_4|\|Q_0\|_{L^\infty}\|\nabla{Q}\|_{L^2}^2+ C|L_4|\|Q_0\|_{L^\infty} \|\Delta{Q}\|_{L^2}( \|\Delta{Q}\|_{L^2}+\|Q\|_{L^2})
\non\\
&\qquad +2|L_4|\|\nabla{Q}_0\|_{L^4}\|\nabla{Q}\|_{L^4}\|\Delta{Q}\|_{L^2}+\|G\|_{L^2}(\|Q\|_{L^2}+\|\Delta{Q}\|_{L^2})\\
&\quad
\leq C|L_4|\|Q_0\|_{L^\infty}(\|\nabla Q\|_{L^2}^2+\|\Delta{Q}\|_{L^2}^2\big)+ C|L_4|\|Q_0\|_{L^\infty}\|Q\|_{L^2}^2 \non\\
&\qquad +C\|Q_0\|_{L^\infty}^\frac12\|\Delta Q_0\|_{L^2}^\frac12\|\nabla{Q}\|_{L^2}^\frac12(\|\Delta{Q}\|_{L^2}^\frac12+\|Q\|_{L^2}^\frac12)\|\Delta{Q}\|_{L^2}\non\\
&\qquad +\|G\|_{L^2}(\|Q\|_{L^2}+\|\Delta{Q}\|_{L^2})\\
&\quad \leq \frac{\zeta}{2}\big(\|\nabla Q\|_{L^2}^2+ \|\Delta{Q}\|_{L^2}^2\big)+C\big(\|Q\|_{L^2}^2+\|\nabla{Q}\|_{L^2}^2\big)+C\|G\|_{L^2}^2,
\end{align*}
where in above estimate, we employed \eqref{eq:1.29} with sufficiently small $C_1$ (but only depending on $\T$). By Gronwall's lemma, we get
\begin{equation}
\|Q\|_{L^\infty(0,T;H^1(\T))}+ \|Q\|_{L^2(0,T;H^2(\T))} \leq C(\|Q_0\|_{H^1} + \|G\|_{L^2(0, T; L^2(\T))}).\label{HHe12}
\end{equation}
Next, testing \eqref{Q equp} with $\Delta^2{Q}$, using \eqref{eq:1.30}, \eqref{eq:1.31} and the elliptic estimates \eqref{H2}--\eqref{H3}, we deduce that
\begin{align*}
&\frac12\frac{d}{dt}\|\Delta{Q}\|^2+\zeta\|\nabla\Delta{Q}\|^2\\
&\quad =-2L_4\int_{\T}(Q_{ij,l}(Q_0)_{lk})_{,km}(\Delta{Q}_{ij})_{,m}dx-\int_{\T}G_{ij,k}(\Delta{Q}_{ij})_{,k}dx\\
&\quad \leq
2|L_4|\|Q_0\|_{L^\infty}\| Q\|_{H^3}\|\nabla\Delta{Q}\|_{L^2}
+2|L_4|\|\nabla{Q}_0\|_{L^4}\|Q\|_{W^{2,4}}\|\nabla\Delta{Q}\|_{L^2}
\\
&\qquad +2|L_4|\|Q_0\|_{H^2}\|\nabla{Q}\|_{L^\infty}\|\nabla\Delta{Q}\|_{L^2}
+\|\nabla{G}\|_{L^2}\|\nabla\Delta{Q}\|_{L^2}\\
&\quad \leq
C|L_4|\|Q_0\|_{L^\infty}\|\nabla\Delta{Q}\|_{L^2}(\|\nabla\Delta{Q}\|_{L^2}+\|Q\|_{L^2})\non\\
&\qquad +C|L_4|\|\nabla{Q}_0\|_{L^4}(\|\Delta{Q}\|_{L^2}+\|Q\|_{L^2})^\frac12(\|\nabla\Delta{Q}\|_{L^2}+\|Q\|_{L^2})^\frac12\|\nabla\Delta{Q}\|_{L^2}
\\
&\qquad +C|L_4|\|Q_0\|_{H^2}\|\nabla{Q}\|_{L^2}^\frac12(\|\nabla\Delta{Q}\|_{L^2} +\|Q\|_{L^2})^\frac12\|\nabla\Delta{Q}\|_{L^2}+\|\nabla{G}\|_{L^2}\|\nabla\Delta{Q}\|_{L^2}\\
&\quad \leq\frac{\zeta}{2}\|\nabla\Delta{Q}\|_{L^2}^2+C\|\Delta Q\|_{L^2}^2+C(\|\nabla{G}\|_{L^2}^2+ \|Q\|_{H^1}^2),
\end{align*}
where in order to obtain the last inequality, we again require $C_1$ to be suitably small. Using Gronwall's lemma and the estimate \eqref{HHe12}, we obtain
\begin{equation*}
\|\Delta Q\|_{L^\infty(0,T;L^2(\T))}+ \|\nabla \Delta Q\|_{L^2(0,T;L^2(\T))} \leq C(\|Q_0\|_{H^2} + \|G\|_{L^2(0, T; H^1(\T))}).
\end{equation*}
Finally, keeping above estimates in mind, using the elliptic estimates and a comparison argument for $\partial_t Q$, we arrive at the conclusion \eqref{eq:1.15}.
\end{proof}

As a corollary, we can easily derive some estimates on the induced nonlinear stress terms:
\begin{lemma}\label{lemma-energy-estimate-2}
Let $Q$ be the solution to problem \eqref{Q equp}--\eqref{ICp} given by Lemma \ref{lemma-energy-estimate-1}. Concerning the stress tensors $\sigma^s$, $\sigma^a$ that are defined via \eqref{tensors}, we have
\begin{equation}
\|\nabla\cdot(\sigma^a+\sigma^s)\|_{L^2(0,T; L^2(\T))}\leq C(\|Q_0\|_{H^2}, \|G\|_{L^2(0,T;H^1(\T))}).
\label{H2H3-integrability}
\end{equation}
\end{lemma}
\begin{proof}
By virtue of \eqref{tensors}, H\"older's inequality and the Sobolev embedding theorem, we obtain
\begin{align*}
\|\nabla\cdot\sigma^s\|_{L^2}
&\leq 4(|L_1|+|L_2|+|L_3|)\|\nabla{Q}\|_{L^\infty}\|Q\|_{H^2}+2|L_4|\|\nabla{Q}\|_{L^6}^3\non\\
&\quad +4|L_4|\|Q\|_{L^\infty}\|\nabla{Q}\|_{L^2}\|Q\|_{H^2}\\
&\leq
C(\|Q\|_{H^2}\|Q\|_{H^3}+\|Q\|_{H^2}^3+\|Q\|_{H^1}\|Q\|_{H^2}^2)\\
&\leq C\|Q\|_{H^2}\|Q\|_{H^3}+ C\|Q\|_{H^2}^3,
\end{align*}
which together with \eqref{eq:1.15} yields that
$$
 \|\nabla\cdot\sigma^s\|_{L^2(0,T; L^2(\T))} \leq C(\|Q_0\|_{H^2}, \|G\|_{L^2(0,T;H^1(\T))}).
$$
Recalling the expression \eqref{eqnQExpansion}, we see that $\|\nabla\cdot\sigma^a\|_{L^2(0,T; L^2(\T))}$ can be estimated in a similar manner. The proof is complete.
\end{proof}

\medskip
\noindent\textbf{Step 3}.  Construction of a nonlinear mapping $\mathcal{Y}$.

\medskip

For arbitrary but fixed $T>0$, we proceed to define a mapping $\mathcal{Y}$ on $L^2(0,T;H^1(\T,\mathcal{S}_0^{(2)}))$ based on the results in Step 2.
To this end, for any given matrix-valued function $G\in L^2(0,T;H^1(\T,\mathcal{S}_0^{(2)}))$, thanks to Lemma \ref{lemma-energy-estimate-1}, problem \eqref{Q equp}--\eqref{ICp}
is uniquely solvable and its solution $Q=Q[G]$ satisfies \eqref{regQp}. Then by Lemma \ref{lemma-energy-estimate-2}, we also have
$\nabla\cdot (\sigma^a(Q)+\sigma^s(Q))\in L^2(0,T;L^2(\T))$, which enables us to apply Lemma \ref{lemma1} to conclude that problem \eqref{eq:1.12a}--\eqref{eq:1.12d} admits a unique global weak solution
$\bu=\bu[G]$ with the regularity property \eqref{regup}. Besides, by Lemma \ref{lemma1}, the solution $\bu$ of problem \eqref{eq:1.12a}--\eqref{eq:1.12d} fulfills
\begin{equation}
\label{eq:1.03}
{\|\bu\|_{L^\infty(0,T; L^2(\T))\cap L^2(0,T;H^{k}(\T))}\leq C\big(\|\nabla\cdot(\sigma^a+\sigma^s)\|_{L^2(0,T;L^2(\T))}+\|\bu_0\|_{L^2}\big)}.
\end{equation}
Hence, it follows from Lemma \ref{lemma-energy-estimate-2}, \eqref{eq:1.15} and \eqref{H2H3-integrability} that
\begin{align}
&\|\bu\|_{{L}^\infty(0,T;L^2(\T))\cap L^2(0,T;H^k(\T))}+ \|Q\|_{L^\infty(0,T;H^2(\T))\cap L^2(0,T;H^3(\T))}\non\\
&\quad \leq C(\|G\|_{L^2(0,T;H^1(\T))}, \| \bu_0\|_{L^2}, \|Q_0\|_{H^2}).\label{eq:1.03a}
\end{align}

Next, we verify that
$$
\mathcal{G}(\bu,Q)+2L_4(Q_{ij,l }Q_{lk})_{,k}-2L_4(Q_{ij,l}(Q_0)_{l k})_{,k}\in L^2(0,T;H^{1}(\T,\mathcal{S}_0^{(2)})),
$$
where $\mathcal{G}$ is given by \eqref{eq:1.02}. First, since $Q\in \mathcal{S}_0^{(2)}$ a.e. in $\T\times(0,T)$,
it is straightforward to check that $\mathcal{G} \in \mathcal{S}_0^{(2)}$ a.e. in $\T\times(0,T)$.
In view of \eqref{eq:1.02}, we deduce from \eqref{eq:1.15}, \eqref{eq:1.03a}, H\"older's inequality and the Sobolev embedding theorem that
\begin{align}
&\|\mathcal{G}(\bu,Q)\|_{H^1}\non\\
&\quad \leq C\big(\|\bu\cdot\nabla{Q}\|_{H^1}+\|\nabla{\bu}Q\|_{H^1}+\|(1+\tr(Q^2))Q\|_{H^1}+\|\nabla{Q}\nabla{Q}\|_{H^1}\big)\non\\
&\quad \leq
C\|\bu\|_{L^4}\|\nabla{Q}\|_{L^4}+C\|\nabla{\bu}\|_{L^4}\|Q\|_{W^{1,4}}+C\|\bu\|_{L^\infty}\|Q\|_{H^2}+C\|\bu\|_{H^2}\|Q\|_{L^\infty}\non\\
&\qquad+C(1+\|Q\|_{L^\infty}^2)\|Q\|_{H^1}+C\|\nabla{Q}\|_{L^4}^2+C\|\nabla{Q}\|_{L^4}\|Q\|_{W^{2,4}}\non\\
&\quad \leq
C(\|G\|_{L^2(0,T;H^1(\T))},\|Q_0\|_{H^2})\Big(1+\|\bu\|_{W^{1,4}}+\|\bu\|_{L^\infty}+\|\bu\|_{H^2}+\|Q\|_{H^3}^\frac12\Big)\non\\
&\quad \leq
C(\|G\|_{L^2(0,T;H^1(\T))},\|Q_0\|_{H^2})\Big(1+\|\bu\|_{L^2}^{\frac{k-2}{k}}\|\bu\|_{H^k}^{\frac2k}+\|Q\|_{H^3}^\frac12\Big)\non\\
&\quad \leq
C(\|G\|_{L^2(0,T;H^1(\T))},\|\bu_0\|_{L^2},\|Q_0\|_{H^2})\Big(1+\|\bu\|_{H^k}^\frac{2}{k}+\|Q\|_{H^3}^\frac12\Big),\label{estG}
\end{align}
which together with \eqref{eq:1.03a} further implies $\mathcal{G}(\bu,Q)\in L^2(0,T; H^1(\T))$.

It remains to estimate
$2|L_4|\|(Q_{ij,l}Q_{lk})_{,k}-(Q_{ij,l}(Q_0)_{lk})_{,k}\|_{L^2(0,T;H^1(\T))}$.
Note that
\begin{align}\label{key-estimate}
2|L_4|&\|(Q_{ij,l}Q_{lk})_{,k}-(Q_{ij,l}(Q_0)_{lk})_{,k}\|_{L^2(0,T;H^1(\T))}\non\\
&\leq C\|Q\|_{L^2(0,T;H^3(\T))}\|Q-Q_0\|_{L^{\infty}(0,T;L^\infty(\T))}\non\\
&\quad +C\|Q\|_{L^2(0,T;W^{2,4}(\T))}\|Q-Q_0\|_{L^{\infty}(0,T;W^{1,4}(\T))}\non\\
&\quad +C\|\nabla{Q}\|_{L^2(0,T;L^\infty(\T))}\|\nabla (Q-Q_0)\|_{L^{\infty}(0,T;H^1(\T))}\non\\
&:=\sum_{i=1}^3 J_i,
\end{align}
with obvious notation. Using Agmon's inequality in two dimensions and the interpolation inequality \eqref{Gel} (keeping in mind that $(Q-Q_0)|_{t=0}=0$),
we have
\begin{align*}
 J_1&\leq C\|Q\|_{L^2(0,T;H^3(\T))}\|Q-Q_0\|^{\frac12}_{L^\infty(0,T;L^2(\T))}\|Q-Q_0\|_{L^\infty(0,T;H^2(\T))}^{\frac12}\\
    &\leq C(\|G\|_{L^2(0,T;H^1(\T))},\|Q_0\|_{H^2})\|Q-Q_0\|^{\frac12}_{L^\infty(0,T;H^1(\T))}\\
    &\leq C(\|G\|_{L^2(0,T;H^1(\T))},\|Q_0\|_{H^2})\|Q-Q_0\|^{\frac14}_{L^2(0,T;H^2(\T))}\|Q-Q_0\|^{\frac14}_{H^1(0,T;L^2(\T))}\\
    &\leq C(\|G\|_{L^2(0,T;H^1(\T))},\|Q_0\|_{H^2})\|Q-Q_0\|^{\frac14}_{L^\infty(0,T;H^2(\T))} T^{\frac18}\\
    &\leq C(\|G\|_{L^2(0,T;H^1(\T))},\|Q_0\|_{H^2}) T^{\frac18}.
\end{align*}
Next, it follows from the interpolation inequality and \eqref{eq:1.15} that
\begin{align*}
J_2 &\leq C\|Q\|_{L^\infty(0,T; H^2(\T))}^\frac12\|Q\|_{L^1(0,T;H^3(\T))}^{\frac12}\|Q-Q_0\|_{L^\infty(0,T;H^2(\T))}\non\\
    &\leq C(\|G\|_{L^2(0,T;H^1(\T))},\|Q_0\|_{H^2})\|Q\|_{L^1(0,T;H^3(\T))}^{\frac12}\non\\
    &\leq C(\|G\|_{L^2(0,T;H^1(\T))},\|Q_0\|_{H^2})\|Q\|_{L^2(0,T;H^3(\T))}^{\frac12} T^{\frac14}\\
    &\leq C(\|G\|_{L^2(0,T;H^1(\T))},\|Q_0\|_{H^2}) T^{\frac14},
\end{align*}
and in a similar manner,
\begin{align*}
J_3&\leq C\|\nabla Q\|_{L^\infty(0,T; L^2(\T))}^\frac12\|\nabla Q\|_{L^1(0,T;H^2(\T))}^{\frac12}\|Q-Q_0\|_{L^{\infty}(0,T;H^2(\T))}\\
   &\leq C(\|G\|_{L^2(0,T;H^1(\T))},\|Q_0\|_{H^2}) T^{\frac14}.
\end{align*}
As a consequence, we obtain
\begin{align}\label{key-estimateaa}
&2|L_4|\|(Q_{ij,l}Q_{lk})_{,k}-(Q_{ij,l}(Q_0)_{lk})_{,k}\|_{L^2(0,T;H^1(\T))}\non\\
&\quad \leq C(\|G\|_{L^2(0,T;H^1(\T))},\|Q_0\|_{H^2})(T^{\frac18}+T^{\frac14}),
\end{align}
which implies $2L_4(Q_{ij,l }Q_{lk})_{,k}-2L_4(Q_{ij,l}(Q_0)_{l k})_{,k}\in L^2(0,T;H^{1}(\T))$.

Therefore, we see that the following nonlinear mapping defined by
\begin{align}\label{def of F}
\mathcal{Y}: L^2(0,T;H^{1}(\T,\mathcal{S}_0^{(2)})) &\to L^2(0,T;H^{1}(\T,\mathcal{S}_0^{(2)}))\non\\
G&\mapsto \mathcal{Y}(G)=\mathcal{G}(\bu, Q)+2L_4 (Q_{ij,l }Q_{l k})_{,k}-2L_4 (  Q_{ij,l }(Q_0)_{l k})_{,k}
\end{align}
is well defined.

\bigskip

\noindent\textbf{Step 4}. The mapping $\mathcal{Y}$ is a contraction for sufficiently small $T>0$.

\medskip
To construct local-in-time solutions to problem \eqref{NSEa}--\eqref{ICa}, it suffices to show that $\mathcal{Y}$ is a contraction on certain closed ball
\begin{equation}\label{eq:1.32}
  B_{T,R}:=\big\{G\in L^2(0,T;H^{1}(\T,\mathcal{S}_0^{(2)}))\,:\ \|G\|_{L^2(0,T;H^{1}(\T))}\leq R\big\}.
\end{equation}
Here, we simply take
$$R=\|\bu_0\|_{L^2} + \|Q_0\|_{H^2}+1.$$

We first show that for certain sufficiently small time $T_0>0$
(depending on $R$), $\mathcal{Y}$ maps $B_{T_0,R}$ into itself. By
\eqref{estG}, we can find $T_1>0$ sufficiently small such that (recalling that $k\geq 4$)
\begin{align}
& \|\mathcal{G}(\bu,Q)\|_{L^2(0,T_1;H^1(\T))}\non\\
&\quad \leq C(R,\|\bu_0\|_{L^2},\|Q_0\|_{H^2})\Big(T_1^\frac12 +\|\bu\|_{L^1(0,T_1;H^k(\T))}^\frac12+\|Q\|_{L^1(0,T_1; H^3(\T))}^\frac12\Big)\non\\
&\quad \leq C(R,\|\bu_0\|_{L^2},\|Q_0\|_{H^2})\Big(T_1^\frac12 + T_1^\frac14\|\bu\|_{L^2(0,T_1;H^k(\T))}^\frac12+T_1^\frac14\|Q\|_{L^2(0,T_1; H^3(\T))}^\frac12\Big)  \leq\frac{R}{2}.\label{R-estimate-1}
\end{align}
On the other hand, due to \eqref{key-estimateaa}, there exists a small $T_2>0$ such that
\begin{align}
&2|L_4|\|(Q_{ij,l}Q_{lk})_{,k}-( Q_{ij,l}(Q_0)_{lk})_{,k}\|_{L^2(0,T_2;H^1(\T))} \leq  C(R,\|Q_0\|_{H^2})(T_2^{\frac18}+T_2^{\frac14})\leq \frac{R}{2}.\label{R-estimate-2}
\end{align}
Thus, we conclude from \eqref{R-estimate-1} and \eqref{R-estimate-2} that for $T_0=\min\{T_1, T_2\}$, the nonlinear mapping
$\mathcal{Y}$ maps $B_{T_0,R}$ into itself.

Next, we show that $\mathcal{Y}$ is actually a contraction. For any $G_i\in B_{T,R}$ ($i=1, 2$), let $(\bu_i, Q_i)$ be the corresponding
solution to the following problem with $G=G_i$:
\begin{equation}\label{equ-1-2}
\begin{cases}
 \p_t \bu +\mathcal{L}_\delta \bu-\nu\Delta{\bu}+\nabla P = - \bu\cdot\nabla{\bu}+\nabla\cdot(\sigma^a+\sigma^s),\quad \forall\, (x,t) \in \T\times(0,T),\\
 \nabla \cdot \bu=0,\qquad\qquad\qquad\qquad\qquad\qquad\qquad\qquad\qquad\quad\;\, \forall\, (x,t) \in \T\times(0,T),\\
 \p_t Q -\zeta \Delta Q-2L_4 (Q_{ij,l}(Q_0)_{lk})_{,k}=G,\quad\qquad\qquad\quad\;\;\;\ \,\forall\, (x,t) \in \T\times(0,T),\\
 \bu(x+e_i,t) = \bu(x,t), \quad Q(x+e_i,t)=Q(x,t), \quad\qquad\;\;\ \,\forall\, (x,t) \in \T\times(0,T),\\
\bu(x,0)=\bu_0(x),\;\; \text{with}\ \nabla \cdot \bu_0=0,\;\;
Q(x,0)=Q_0(x),\quad\,\, \forall\, x \in \T,
\end{cases}
\end{equation}
where $\sigma^a$ and $\sigma^s$ are given by \eqref{tensora}, \eqref{tensors}, respectively.
Denote
$$\hat{\bu}=\bu_1-\bu_2,\quad \hat{Q}=Q_1-Q_2,\quad \hat{P}=P_1-P_2,\quad \hat{G}=G_1-G_2.$$
 Then the difference functions $(\hat{\bu}, \hat{P}, \hat{Q})$ satisfy
\begin{equation}\label{equ-contraction}
\begin{cases}
 \p_t \hat{\bu} +\mathcal{L}_\delta\hat{\bu}-\nu\Delta\hat{\bu}+\nabla\hat{P}=-\bu_1\cdot \nabla\hat{\bu}-\hat{\bu}\cdot \nabla{\bu}_2+\nabla\cdot(\hat{\sigma}^a+\hat{\sigma}^s),\quad \forall\, (x,t) \in \T\times(0,T),\\
 \nabla \cdot \hat{\bu}=0,\qquad \qquad \qquad \qquad \qquad \qquad \qquad \qquad \qquad \qquad \qquad \qquad \ \,\forall\, (x,t) \in \T\times(0,T),\\
 \p_t \hat{Q} -\zeta\Delta\hat{Q}-2L_4(\hat{Q}_{ij,l}(Q_0)_{lk})_{,k}=\hat{G}, \quad\qquad\qquad\qquad\qquad\qquad\quad\ \forall\, (x,t) \in \T\times(0,T),\\
\hat{\bu}(x+e_i,t) = \hat{\bu}(x,t), \quad \hat{Q}(x+e_i,t)=\hat{Q}(x,t), \quad\qquad\qquad\qquad\quad\;\;\ \,\forall\, (x,t) \in \T\times(0,T),\\
\hat{\bu}(x,0)=0,\quad Q(x,0)=0,\quad\qquad\qquad\qquad\qquad\qquad\qquad\qquad\qquad\ \, \forall\, x \in \T,
\end{cases}
\end{equation}
where
\begin{align*}
  &\hat{\sigma}^a=\sigma^a(Q_1)-\sigma^a(Q_2),\quad \hat{\sigma}^s=\sigma^s(Q_1)-\sigma^s(Q_2).
\end{align*}
First, by Lemma \ref{lemma-energy-estimate-1}, we have the estimate for $\hat{Q}$:
\begin{equation}\label{contraction-estimate-1}
  \|\hat{Q}\|_{H^1(0,T;H^{1}(\T))\cap C([0,T]; H^2(\T))\cap L^2(0,T;H^{3}(\T))}\leq C\|\hat{G}\|_{L^2(0,T;H^1(\T))}.
\end{equation}
Denote for simplicity $\widetilde{\HH}=\HH+\lambda\mathbb{I}+\mu-\mu^T$. Then we deduce from the estimate \eqref{eq:1.15} for $Q_1$, $Q_2$ that
\begin{align}
&\|\nabla\cdot\hat{\sigma}^a\|_{L^2}\non\\
&\quad\leq  \|\nabla\cdot(\hat{Q}\widetilde{\mathcal{H}}(Q_1))\|_{L^2}+ \|\nabla \cdot(Q_2(\widetilde{\mathcal{H}}(Q_1)-\widetilde{\mathcal{H}}(Q_2)))\|_{L^2}
\non\\
&\qquad +\|\nabla\cdot(\widetilde{\mathcal{H}}(Q_1)\hat{Q})\|_{L^2}+ \|\nabla \cdot((\widetilde{\mathcal{H}}(Q_1)-\widetilde{\mathcal{H}}(Q_2))Q_2)\|_{L^2}\non\\
&\quad \leq C\|\hat{Q}\|_{L^\infty}\|\nabla \widetilde{\mathcal{H}}(Q_1)\|_{L^2}
+ C \|\nabla\hat{Q}\|_{L^\infty}\|\widetilde{\mathcal{H}}(Q_1)\|_{L^2}\non\\
&\qquad + C\|\nabla Q_2\|_{L^\infty}\|\widetilde{\mathcal{H}}(Q_1)-\widetilde{\mathcal{H}}(Q_2)\|_{L^2}
+ C\|Q_2\|_{L^\infty}\|\nabla(\widetilde{\mathcal{H}}(Q_1)-\widetilde{\mathcal{H}}(Q_2))\|_{L^2}\non\\
&\quad \leq
C\|\hat{Q}\|_{L^\infty}\left(\|\nabla \Delta Q_1\|_{L^2}+\|\nabla Q_1\|_{L^\infty}\| Q_1\|_{H^2}+\|Q_1\|_{L^\infty}\|Q_1\|_{H^3}\right) \non\\
&\qquad +C\|\hat{Q}\|_{L^\infty}(1+\|Q_1\|_{L^\infty}^2)\|\nabla Q_1\|_{L^2}+C \|\nabla \hat{Q}\|_{L^\infty}(\|Q_1\|_{L^2}+\|Q_1\|_{L^6}^3)\non\\
&\qquad +C \|\nabla \hat{Q}\|_{L^\infty}\left(\|\Delta Q_1\|_{L^2}+\|\nabla Q_1\|_{L^4}^2+\|Q_1\|_{L^\infty}\| Q_1\|_{H^2}\right)\non\\
&\qquad +C \|\nabla Q_2\|_{H^2}\big[\|\Delta \hat{Q}\|_{L^2}+(\|\nabla Q_1\|_{L^4}+\|\nabla Q_2\|_{L^4})\|\nabla \hat{Q}\|_{L^4}\big]\non\\
&\qquad +C \|\nabla Q_2\|_{H^2}\big(\|\hat{Q}\|_{L^\infty}\|Q_1\|_{H^2}+\|Q_2\|_{L^\infty}\|\hat{Q}\|_{H^2}\big)\non\\
&\qquad +C \|\nabla Q_2\|_{H^2}(1+\|Q_1\|_{L^\infty}^2+\|Q_2\|_{L^\infty}^2)\|\hat{Q}\|_{L^2}\non\\
&\qquad +C \|Q_2\|_{L^\infty}\big[\|\nabla \Delta \hat{Q}\|_{L^2}+ \|\hat{Q}\|_{W^{2,4}}(\|\nabla Q_1\|_{L^4}+\|\nabla Q_2\|_{L^4})\big]\non\\
&\qquad +C \|Q_2\|_{L^\infty}\big[\|\nabla \hat{Q}\|_{L^\infty}(\|Q_1\|_{H^2}+\| Q_2\|_{H^2}) + \|Q_1\|_{H^3}\|\hat{Q}\|_{L^\infty}\big]\non\\
&\qquad +C \|Q_2\|_{L^\infty}^2\|\hat{Q}\|_{H^3} + C \|Q_2\|_{L^\infty} (1+\|Q_1\|_{L^\infty}^2+\|Q_2\|_{L^\infty}^2)\|\nabla \hat{Q}\|_{L^2}\non\\
&\qquad +C \|Q_2\|_{L^\infty}(\|Q_1\|_{L^\infty}+\|Q_2\|_{L^\infty})(\|\nabla Q_1\|_{L^2}+\|\nabla Q_2\|_{L^\infty})\|\hat{Q}\|_{L^\infty}\non\\
&\quad \leq C(R,\|Q_0\|_{H^2})\|\hat{Q}\|_{H^3}+C(R,\|Q_0\|_{H^2})\big(\|Q_1\|_{H^3}+\|Q_2\|_{H^3}\big)\|\hat{Q}\|_{H^2},\non
\end{align}
and henceforth
\begin{align}\label{contraction-estimate-2}
&\|\nabla\cdot\hat{\sigma}^a\|_{L^2(0,T;L^2(\T))}\non\\
&\quad \leq C(R,\|Q_0\|_{H^2})\|\hat{Q}\|_{L^2(0,T;H^3(\T))}\non\\
&\qquad +C(R,\|Q_0\|_{H^2})\big(\|Q_1\|_{L^2(0,T;H^3(\T))}+\|Q_2\|_{L^2(0,T;H^3(\T))}\big)\|\hat{Q}\|_{L^\infty(0,T;H^2(\T))}\non\\
&\quad \leq C(R,\|Q_0\|_{H^2})\|\hat{G}\|_{L^2(0,T;H^1(\T))}.
\end{align}
Analogously, we have
\begin{align}
&\|\nabla\cdot\hat{\sigma}^s\|_{L^2}\non\\
&\quad \leq C\big(\|Q_1\|_{H^2}+\|Q_2\|_{H^2}\big)\|\nabla \hat{Q}\|_{L^\infty}+ C\big(\|\nabla Q_1\|_{L^\infty}+\|\nabla Q_2\|_{L^\infty}\big)\|\hat{Q}\|_{H^2}\non\\
&\qquad + C|L_4|\|\hat{Q}\|_{L^\infty}\|Q_1\|_{H^2}\|\nabla Q_1\|_{L^\infty} +C|L_4|\|\nabla \hat{Q}\|_{L^\infty}\big(\|\nabla Q_1\|_{L^4}^2+\|\nabla Q_2\|_{L^4}^2\big)\non\\
&\qquad + C|L_4|\|Q_2\|_{L^\infty}\|\nabla \hat{Q}\|_{L^\infty}\big(\|Q_1\|_{H^2}+\|Q_2\|_{H^2}\big)\non\\
&\qquad + C|L_4|\|Q_2\|_{L^\infty}\|\hat{Q}\|_{H^2}\big(\|\nabla Q_1\|_{L^\infty}+\|\nabla Q_2\|_{L^\infty}\big)\non\\
&\quad \leq C(R,\|Q_0\|_{H^2})\|\hat{Q}\|_{H^3}+C(R,\|Q_0\|_{H^2})\big(\|Q_1\|_{H^3}+\|Q_2\|_{H^3}\big)\|\hat{Q}\|_{H^2},\non
\end{align}
which yields
\begin{equation}\label{contraction-estimate-3}
\|\nabla\cdot\hat{\sigma}^s\|_{L^2(0,T;L^2(\T))}\leq C(R,\|Q_0\|_{H^2})\|\hat{G}\|_{L^2(0,T;H^1(\T))}.
\end{equation}
Back to the first equation of \eqref{equ-contraction}, we infer from
Lemma \ref{lemma1} (i.e., \eqref{eq:1.16}) together with the estimates \eqref{contraction-estimate-2} and \eqref{contraction-estimate-3} that
\begin{align}
\|\hat{\bu}\|_{C([0,T];L^2(\T))\cap L^2(0,T;H^{k}(\T))}&\leq
C\|\nabla\cdot(\hat{\sigma}^a+\hat{\sigma}^s)\|_{L^2(0,T;L^2(\T))}\non\\
&\leq C(R,\|Q_0\|_{H^2})\|\hat{G}\|_{L^2(0,T;H^1(\T))}.
\label{contraction-estimate-4}
\end{align}
As a consequence, keeping in mind the assumption $k\geq 4$, we deduce that
\begin{align*}
&\|\mathcal{G}(\bu_1,Q_1)-\mathcal{G}(\bu_2,Q_2)\|_{H^1}\\
&\quad \leq C\|\hat{\bu}\|_{W^{1,4}}\|\nabla{Q}_1\|_{L^4}+C\|\hat{\bu}\|_{L^\infty}\| Q_1\|_{H^2}+C\|\nabla {\bu}_2\|_{L^2}\|\nabla\hat{Q}\|_{L^\infty}
\non\\
&\qquad +C\|\bu_2\|_{L^4}\|\nabla \hat{Q}\|_{W^{1,4}}
+C\|\hat{\bu}\|_{H^2}\|Q_1\|_{L^\infty}
+ C\|\nabla \bu_2\|_{H^1}\|\hat{Q}\|_{L^\infty}\\
&\qquad +C\| Q_1\|_{W^{1,\infty}}\|\nabla \hat{\bu}\|_{L^2}
+C(1+\|Q_1\|_{L^\infty}^2+\|Q_2\|_{L^\infty}^2)\|\hat{Q}\|_{H^1} \\
&\qquad +C\|\nabla\hat{Q}\|_{L^\infty}(\|\nabla Q_1\|_{H^1}+\|\nabla Q_2\|_{H^1})
+C\|\hat{Q}\|_{W^{2,4}}(\|\nabla{Q}_1\|_{L^4}+\|\nabla{Q}_2\|_{L^4})\\
&\leq
C(R,\|\bu_0\|_{L^2},\|Q_0\|_{H^2})\Big(\|\hat{\bu}\|_{L^2}^{\frac{2k-3}{2k}}\|\hat{\bu}\|_{H^k}^{\frac{3}{2k}}
+\|\hat{\bu}\|_{L^2}^{\frac{k-1}{k}}\|\hat{\bu}\|_{H^k}^{\frac{1}{k}}
+\|\bu_2\|_{H^k}^{\frac1k}\|\nabla \hat{Q}\|_{L^2}^\frac12\|\nabla \hat{Q}\|_{H^2}^\frac12\\
&\qquad +\|\bu_2\|_{H^k}^{\frac{1}{2k}}\|\nabla \hat{Q}\|_{H^1}^\frac12\|\nabla \hat{Q}\|_{H^2}^\frac12
+\|\hat{\bu}\|_{L^2}^{\frac{k-2}{k}}\|\hat{\bu}\|_{H^k}^{\frac{2}{k}}
+\|\bu_2\|_{H^k}^{\frac{2}{k}}\|\hat{Q}\|_{H^2} \non\\
&\qquad + \|Q_1\|_{H^3}^\frac12\|\hat{\bu}\|_{L^2}^{\frac{k-1}{k}}\|\hat{\bu}\|_{H^k}^{\frac{1}{k}}
+\|\hat{Q}\|_{H^1}+\|\hat{Q}\|_{H^2}^\frac12\|\hat{Q}\|_{H^3}^\frac12\Big),
\end{align*}
which along with \eqref{eq:1.03a} and \eqref{contraction-estimate-4} further implies that
\begin{align}
&\|\mathcal{G}(\bu_1,Q_1)-\mathcal{G}(\bu_2,Q_2)\|_{L^2(0,T;H^1(\T))}\non\\
&\quad \leq
C(R,\|\bu_0\|_{L^2},\|Q_0\|_{H^2})\Big(T^{\frac{2k-3}{4k}}\|\hat{\bu}\|_{L^\infty(0,T;L^2(\T))}^\frac{2k-3}{2k}\|\hat{\bu}\|_{L^2(0,T;H^k(\T))}^\frac{3}{2k}
\non\\
&\qquad +T^{\frac{k-1}{2k}}\|\hat{\bu}\|_{L^\infty(0,T;L^2(\T))}^\frac{k-1}{k}\|\hat{\bu}\|_{L^2(0,T;H^k(\T))}^\frac{1}{k} \non\\
&\qquad +T^{\frac{k-2}{4k}}\|\bu_2\|_{L^2(0,T;H^k(\T))}^{\frac{1}{k}}\|\hat{Q}\|_{L^\infty(0,T;H^1(\T))}^\frac12\|\hat{Q}\|_{L^2(0,T;H^3(\T))}^\frac12\non\\
&\qquad
+ T^{\frac{k-1}{4k}}\|\bu_2\|_{L^2(0,T;H^k(\T))}^{\frac{1}{2k}}\|\hat{Q}\|_{L^\infty(0,T;H^2(\T))}^\frac12\|\hat{Q}\|_{L^2(0,T;H^3(\T))}^\frac12\non\\
&\qquad
+ T^{\frac{k-2}{2k}}\|\hat{\bu}\|_{L^\infty(0,T;L^2(\T))}^\frac{k-2}{k}\|\hat{\bu}\|_{L^2(0,T;H^k(\T))}^\frac{2}{k}\non\\
&\qquad + T^{\frac{k-2}{2k}}\|\bu_2\|_{L^2(0,T;H^k(\T))}^{\frac{2}{k}}\|\hat{Q}\|_{L^\infty(0,T;H^2(\T))}\non\\
&\qquad + T^{\frac{k-2}{4k}}\|Q_1\|_{L^2(0,T;H^3(\T))}^\frac12\|\hat{\bu}\|_{L^\infty(0,T;L^2(\T))}^{\frac{k-1}{k}}\|\hat{\bu}\|_{L^2(0,T;H^k(\T))}^{\frac{1}{k}}\non\\
&\qquad
+T^\frac12\|\hat{Q}\|_{L^\infty(0,T;H^1(\T))}
+T^\frac14\|\hat{Q}\|_{L^\infty(0,T;H^2(\T))}^\frac12\|\hat{Q}\|_{L^2(0,T;H^3(\T))}^\frac12\Big)\non\\
&\quad \leq C(R,\|\bu_0\|_{L^2},\|Q_0\|_{H^2})T^\frac{k-2}{4k}\big(1+ T^\frac{k+2}{4k}\big)\|\hat{G}\|_{L^2(0,T;H^1(\T))}.
\label{contraction-estimate-5}
\end{align}
Next, we note that
\begin{align*}
&2|L_4|\big\|\big((Q_1)_{ij,l}(Q_1)_{lk}\big)_{,k}-\big((Q_1)_{ij,l}(Q_0)_{lk}\big)_{,k}\non\\
&\qquad \quad -\big[\big((Q_2)_{ij,l}(Q_2)_{lk}\big)_{,k}-\big((Q_2)_{ij,l}(Q_0)_{lk}\big)_{,k}\big]\big\|_{L^2(0,T;H^1(\T))}\\
&\quad \leq 2|L_4|\big\|\big(\hat{Q}_{ij,l}((Q_1)_{lk}-(Q_0)_{lk})\big)_{,k}\big\|_{L^2(0,T;H^1(\T))}\non\\
&\qquad +2|L_4|\big\|\big((Q_2)_{ij,l}\hat{Q}_{lk}\big)_{,k}\big\|_{L^2(0,T;H^1(\T))}\non\\
&\quad :=J_1+J_2,
\end{align*}
with obvious notation. Recalling that $(Q_i-Q_0)|_{t=0}=0$ for $i=1,2$, then using \eqref{contraction-estimate-1} and a similar argument as for \eqref{key-estimate},
we obtain
\begin{align*}
J_1
&\leq C\|\hat{Q}\|_{L^2(0,T;H^3(\T))}\|Q_1-Q_0\|_{L^{\infty}(0,T;L^\infty(\T))}\non\\
&\quad +C \|\hat{Q}\|_{L^2(0,T;W^{2,4}(\T))} \|Q_1-Q_0\|_{L^{\infty}(0,T;W^{1,4}(\T))}\\
&\quad +C \|\nabla\hat{Q}\|_{L^2(0,T; L^\infty(\T))}\|\nabla (Q_1-Q_0)\|_{L^{\infty}(0,T; H^1(\T))}\\
&\leq C\|Q_1-Q_0\|_{L^{\infty}(0,T;H^2(\T))}^\frac14T^{\frac18}\|\hat{Q}\|_{L^2(0,T;H^3(\T))}\non\\
&\quad +C\|Q_1-Q_0\|_{L^\infty(0,T; H^2(\T))}\|\hat{Q}\|_{L^\infty(0,T;H^2(\T))}^\frac12\|\hat{Q}\|_{L^2(0,T;H^3(\T))}^\frac12T^{\frac14}\non\\
&\quad +C\|Q_1-Q_0\|_{L^\infty(0,T; H^2(\T))}\|\nabla\hat{Q}\|_{L^\infty(0,T; L^2(\T))}^\frac12 \|\nabla\hat{Q}\|_{L^2(0,T; H^2(\T))}^\frac12 T^{\frac14} \\
&\leq CT^{\frac18}(1+T^\frac18)(\|\hat{Q}\|_{L^2(0,T;H^3(\T))}+\|\hat{Q}\|_{L^\infty(0,T;H^2(\T))})\\
&\leq C(R, \|Q_0\|_{H^2})T^{\frac18}(1+T^\frac18)\|\hat{G}\|_{L^2(0,T;H^1(\T))},
\end{align*}
and
\begin{align*}
J_2
&\leq C\|Q_2\|_{L^2(0,T;H^3(\T))}\|\hat{Q}\|_{L^{\infty}(0,T;L^\infty(\T))}\non\\
&\quad +C \|Q_2\|_{L^2(0,T;W^{2,4}(\T))} \|\hat{Q}\|_{L^{\infty}(0,T;W^{1,4}(\T))}\\
&\quad +C \|\nabla Q_2\|_{L^2(0,T; L^\infty(\T))}\|\nabla \hat{Q}\|_{L^{\infty}(0,T; H^1(\T))}\\
&\leq C\|\hat{Q}\|_{L^{\infty}(0,T;H^2(\T))}^\frac34T^{\frac18}\|\hat{Q}\|_{H^1(0,T;L^2(\T))}^\frac14\|Q_2\|_{L^2(0,T;H^3(\T))}\non\\
&\quad +C\|\hat{Q}\|_{L^\infty(0,T; H^2(\T))}\|Q_2\|_{L^\infty(0,T;H^2(\T))}^\frac12\|Q_2\|_{L^2(0,T;H^3(\T))}^\frac12T^{\frac14}\non\\
&\quad +C\|\hat{Q}\|_{L^\infty(0,T; H^2(\T))}\|\nabla Q_2\|_{L^\infty(0,T; L^2(\T))}^\frac12 \|\nabla Q_2\|_{L^2(0,T; H^2(\T))}^\frac12 T^{\frac14} \\
&\leq CT^{\frac18}(1+T^\frac18)(\|\hat{Q}\|_{H^1(0,T;L^2(\T))}+\|\hat{Q}\|_{L^\infty(0,T;H^2(\T))})\\
&\leq C(R, \|Q_0\|_{H^2})T^{\frac18}(1+T^\frac18)\|\hat{G}\|_{L^2(0,T;H^1(\T))}.
\end{align*}
Then we have
\begin{align}
&2|L_4|\big\|\big((Q_1)_{ij,l}(Q_1)_{lk}\big)_{,k}-\big((Q_1)_{ij,l}(Q_0)_{lk}\big)_{,k}\non\\
&\qquad \quad -\big[\big((Q_2)_{ij,l}(Q_2)_{lk}\big)_{,k}-\big((Q_2)_{ij,l}(Q_0)_{lk}\big)_{,k}\big]\big\|_{L^2(0,T;H^1(\T))}\non\\
&\quad \leq C(R, \|Q_0\|_{H^2})T^{\frac18}(1+T^\frac18)\|\hat{G}\|_{L^2(0,T;H^1(\T))}.
\label{contraction-estimate-5aa}
\end{align}
As a consequence, we conclude from \eqref{def of F}, \eqref{contraction-estimate-5}
and \eqref{contraction-estimate-5aa} that there exists a sufficiently small time $T_0'>0$, it holds
\begin{equation*}
\|\mathcal{Y}(G_1)-\mathcal{Y}(G_2)\|_{L^2(0,T_0'; H^1(\T))}\leq \frac12\|G_1-G_2\|_{L^2(0,T_0';H^1(\T))}, \quad \forall\, G_1,\,G_2\in B_{T_0',R}.
\end{equation*}

In summary, we can take $T^*=\min\{T_0,\, T_0'\}$ and apply Banach's fixed point theorem to deduce that
the nonlinear mapping $\mathcal{Y}$ has a unique fixed point $G^*$ in $B_{T^*, R}$ such that $G^*=\mathcal{Y}(G^*)$.
By the definition of $\mathcal{Y}$, this implies that problem \eqref{NSEa}--\eqref{ICa} admits
a unique local solution $(\bu, Q)$ (corresponding to the fixed point $G^*$) satisfying
\begin{align}
   &\bu\in H^1(0,T^*;(H^k_{\sigma}(\T))')\cap C([0,T^*]; L^2_\sigma(\T))\cap L^2(0,T^*;H^k_{\sigma}(\T)),\label{eq:1.33}\\
   &Q\in H^1(0,T^*; H^1(\T))\cap C([0,T^*];H^2(\T))\cap L^2(0,T^*;H^3(\T)),\label{eq:1.33a}
\end{align}
and $Q\in \mathcal{S}_0^{(2)}$ a.e. in $\T\times(0,T)$.

\subsection{Global existence}
In what follows, we proceed to extend the local-in-time solution $(\bu, Q)$ of problem \eqref{NSEa}--\eqref{ICa}
that was constructed above to be a global one.
This goal can be achieved by deriving some uniform in time estimates.

First, it follows from Lemma \ref{prop on maximum principle} that
\begin{align}
 \|Q\|_{C([0, T^*]; L^\infty(\T))} \leq \sqrt{\eta_1},\label{Linfty}
\end{align}
where we recall that the constant $C_1$ in \eqref{eta1} is taken to be suitably small, but it only depends on $\T$.
The above $L^\infty$-estimate will play an important role in deriving global estimates for the solution $(\bu, Q)$.
Next, using an essentially identical argument for Proposition
\ref{proposition on energy law}, we observe that the solution $(\bu, Q)$ problem \eqref{NSEa}--\eqref{ICa} satisfies the following dissipative energy law:
\begin{align}
 & \frac 12\int_{\T} |\bu(t)|^2 dx+  \mathcal{E}(Q(t)) + \int_0^t \int_{\T}(\nu|\nabla \bu|^2+\delta|\mathcal{D}^k \bu|^2)dxds \non\\
 &\qquad +\int_0^t\int_{\T}\tr^2( \HH+\lambda\Id+\mu-\mu^{T})dxds\non\\
 &\quad  = \frac 12\int_{\T} |\bu_0|^2 dx+  \mathcal{E}(Q_0),\quad \forall\, t\in [0,T^*].
  \label{lemma-energy-law}
\end{align}
By the coercivity assumption \eqref{coercivity} and \cite[Lemma C1]{IXZ14}, we have
\begin{align}
\int_{\T} \left(L_1\partial_kQ_{ij}\partial_kQ_{ij}+L_2\partial_jQ_{ik}\partial_kQ_{ij}+L_3\partial_jQ_{ij}\partial_kQ_{ik}\right) dx
\geq \kappa \int_{\T}|\nabla Q|^2dx.\non
\end{align}
Then choose the constant $C_1$ in \eqref{eta1} to be small enough (again only depending on $\T$,
see also \cite[Section 3.2]{IXZ14}), we have
\begin{align}
 \mathcal{E}(Q(t))&\geq   \left(\kappa -|L_4|\|Q(t)\|_{L^\infty}\right) \int_{\T} |\nabla Q(t)|^2dx
 +\frac{c}{4}\int_{\T} \left[\left(\tr(Q^2)+\frac{a}{c}\right)^2-\frac{a^2}{c^2}\right]dx\non\\
 &\geq \frac{\kappa}{2}\int_{\T} |\nabla Q(t)|^2dx-\frac{a^2}{4c}|\T|,\quad \forall\, t\in [0,T^*],
 \label{below}
\end{align}
which implies that the free energy $\mathcal{E}(Q(t))$ is uniformly bounded from below.
From \eqref{Linfty}--\eqref{below}, we infer that
\begin{align}
 & \|\bu\|_{L^\infty(0,T^*;L^2(\T))\cap L^2(0,T^*;H^k(\T))}\leq C,\label{eq:1.07}\\
 & \|Q\|_{L^\infty(0,T^*;H^1(\T))}\leq C.\label{eq:1.06}
\end{align}
It follows from \eqref{lemma-energy-law} that $\HH+\lambda\Id+\mu-\mu^{T}\in L^2(0,T^*; L^2(\T))$.
Hence, using \eqref{Linfty}, for sufficiently small $C_1$, we can apply Lemma \ref{lemma1.0}  to conclude that
\begin{equation}\label{eq:1.06a}
  \|Q\|_{L^2(0,T^*;H^2(\T))}\leq C.
\end{equation}

Next, we derive necessary higher-order estimates for $Q$.
Multiplying \eqref{Q equb} by $\Delta^2 Q$, integrating over $\T$, after integration by parts and using the Cauchy--Schwarz inequality, we get
\begin{align}
&\frac12 \frac{d}{dt} \|\Delta Q\|_{L^2}^2+\zeta\|\nabla\Delta Q\|_{L^2}^2\non\\
&\quad = \int_{\T} \partial_m \left(\mathcal{G}_{ij}(\bu, Q) + 2L_4(Q_{ij,l}Q_{lk})_{,k}\right)\partial_m \Delta Q_{ij} dx \non\\
&\quad \leq \frac{\zeta}{4}\|\nabla\Delta Q\|_{L^2}^2 + C|L_4|\|Q\|_{L^\infty}\|Q\|_{H^3}^2 + C \|Q\|_{W^{2,4}}^2\|\nabla Q\|_{L^4}^2\non\\
&\qquad +C\|\nabla\mathcal{G}(\bu,Q)\|_{L^2}^2.\label{highh2}
\end{align}
Choosing $C_1$ to be small enough, we have
\begin{align}
C|L_4|\|Q\|_{L^\infty}\|Q\|_{H^3}^2&\leq \frac{\zeta}{12}(\|\nabla\Delta Q\|_{L^2}^2 +\|Q\|^2_{L^2}).\non
\end{align}
Besides, by \eqref{eq:1.27}, \eqref{eq:1.30} and the estimate \eqref{eq:1.06}, we see that
\begin{align}
& C \|Q\|_{W^{2,4}}^2\|\nabla Q\|_{L^4}^2\non\\
&\quad \leq C \|Q\|_{H^3}\|Q\|_{H^2}\|\Delta Q\|_{L^2}\|Q\|_{L^\infty}\non\\
&\quad \leq C\|\nabla\Delta Q\|_{L^2}\|\Delta Q\|_{L^2}^2\|Q\|_{L^\infty} + C \|\nabla\Delta Q\|_{L^2}\|\Delta Q\|_{L^2}\|Q\|_{L^2}\|Q\|_{L^\infty}\non\\
&\qquad +C\|\Delta Q\|_{L^2}^2\|Q\|_{L^2}\|Q\|_{L^\infty}+C\|\Delta Q\|_{L^2}\|Q\|_{L^2}^2\|Q\|_{L^\infty}\non\\
&\quad\leq \frac{\zeta}{12}\|\nabla\Delta Q\|_{L^2}^2 + C(1+\|\Delta Q\|_{L^2}^2)\|\Delta Q\|_{L^2}^2+ C.\non
\end{align}
Finally, it holds
\begin{align}
&\|\nabla \mathcal{G}(\bu,Q)\|_{L^2}^2\non\\
&\quad \leq C\|\nabla \bu\|_{L^4}^2\|\nabla Q\|_{L^4}^2+C\|\bu\|_{L^\infty}^2\|Q\|_{H^2}^2+C\|\bu\|_{H^2}^2\|Q\|_{L^\infty}^2\non\\
&\qquad +C(1+\|Q\|_{L^\infty}^2)^2\|Q\|_{H^1}^2+ C\|Q\|_{W^{2,4}}^2\|\nabla Q\|_{L^4}^2 \non\\
&\quad \leq C\|\bu\|_{H^2}^2 \|\Delta Q\|_{L^2}^2 + C\|\bu\|_{H^2}^2+\frac{\zeta}{12}\|\nabla\Delta Q\|_{L^2}^2
+ C(1+\|\Delta Q\|_{L^2}^2)\|\Delta Q\|_{L^2}^2+ C.\label{eq:1.34}
\end{align}
From the above estimates, we deduce from \eqref{highh2} that
\begin{align}
& \frac{d}{dt} \|\Delta Q\|_{L^2}^2+\zeta\|\nabla\Delta Q\|_{L^2}^2\non\\
&\quad \leq  C(1+\|\bu\|_{H^2}^2+\|\Delta Q\|_{L^2}^2)\|\Delta Q\|_{L^2}^2+ C(1+\|\bu\|_{H^2}^2).\non
\end{align}
This inequality together with Gronwall's lemma and the estimates \eqref{eq:1.07}--\eqref{eq:1.06a} yields
\begin{equation}\label{eq:1.08}
   \|Q\|_{L^2(0,T^*;H^3(\T))\cap L^\infty(0,T^*;H^2(\T))} \leq C.
 \end{equation}
At last, by comparison for time derivatives $\p_t\bu$ and $\p_t Q$, we infer from \eqref{eq:1.07}--\eqref{eq:1.06a}, \eqref{eq:1.34} and \eqref{eq:1.08} that
\begin{align}
 & \|\bu\|_{H^1(0,T^*;(H^k_\sigma(\T))')}\leq C,\quad  \|Q\|_{H^1(0,T^*;H^1(\T))}\leq C.
 \label{eq:1.09}
\end{align}

Since the bounds in all the above estimates \eqref{Linfty}, \eqref{eq:1.07}--\eqref{eq:1.06a}
and \eqref{eq:1.08}--\eqref{eq:1.09} only depend on $T>0$ and thus are independent of $T^*$, we are able to extend the (unique)
local solution $(\bu,Q)$ of problem \eqref{NSEa}--\eqref{ICa} to arbitrary time interval $[0,T]$, i.e., it is indeed a global solution.

The proof of Proposition \ref{approx1} is complete.

\section{Proof of the Main Result}
\setcounter{equation}{0}

In this section, we prove Theorem \ref{main-theorem} on the existence and uniqueness of global weak solutions to the
original problem \eqref{NSE}--\eqref{IC}.
\subsection{Global existence for initial data $(\bu_0,Q_0)\in L^2_\sigma(\T)\times H^2(\T)$}

Based on Proposition \ref{approx1}, we can pass to the limit as $\delta\rightarrow 0^+$ in the approximate problem \eqref{NSEa}--\eqref{ICa} to
show the existence of global weak solutions to problem \eqref{NSE}--\eqref{Q equ} with a slightly more regular initial data, i.e., $Q_0\in H^2(\T)$.

\begin{lemma}\label{prop on existence}
Suppose that the assumptions in Proposition \ref{approx1} are satisfied.
For any $(\bu_0,\, Q_0)\in L^2_\sigma(\T)\times H^2(\T)$ with
$\|Q_0\|_{L^\infty}\leq \sqrt{\eta_1}$, problem \eqref{NSE}--\eqref{Q equ} admits a global weak solution $(\bu, Q)$
satisfying
\begin{align*}
&\bu\in H^1(0,T; (H^{1}_\sigma(\T))')\cap C([0,T]; L^2_\sigma(\T))\cap L^2(0,T; H_\sigma^1(\T)),\\
&Q\in H^1(0,T; L^2(\T))\cap C([0,T]; H^1(\T)) \cap L^2(0,T; H^2(\T)).
\end{align*}
Besides, it holds
\begin{equation}\label{eq:1.20b}
 \|Q\|_{L^\infty(0, T;L^\infty(\T))} \leq \sqrt{\eta_1},
\end{equation}
and now the energy identity \eqref{lemma-energy-law} is satisfied with $\delta=0$.
\end{lemma}
\begin{proof}
For the given initial data $(\bu_0, Q_0)$ and an arbitrary fixed $\delta>0$,
according to Proposition \ref{approx1}, problem \eqref{NSEa}--\eqref{ICa} admits
a unique global solution denoted by $(\bu^\delta,\,Q^\delta)$ that satisfies \eqref{rregu}--\eqref{rregQ} and
\begin{equation}\label{eq:1.24}
\|Q^\delta\|_{L^\infty(0,T;L^\infty(\T))}\leq \sqrt{\eta_1}.
\end{equation}
Moreover, we have
\begin{align}
&\int_0^T\langle \p_t \bu^\delta,\mathbf{v}\rangle_{(H^{k}_\sigma)',H^k_\sigma}dt
 -\int_0^T\int_{\T}  \bu^\delta\otimes \bu^\delta : \nabla \mathbf{v} dxdt\non\\
 &\qquad
   +\nu \int_0^T\int_{\T} \nabla{\bu^\delta}:\nabla{\mathbf{v}}dxdt
   +\delta \int_0^T\int_{\T} \mathcal{D}^k{\bu^\delta}\cdot \mathcal{D}^k{\mathbf{v}}dxdt
  \nonumber\\
&\quad = -\int_0^T\int_{\T} \big(\sigma^a(Q^\delta)+\sigma^s(Q^\delta)\big) : \nabla \mathbf{v} dxdt,
\label{eq:1.22}
\end{align}
for any $\mathbf{v}\in L^2(0,T;H^k_{\sigma}(\T))$, and
\begin{align}
&\p_t Q^\delta_{ij}+u^\delta_k {Q}^\delta_{ij,k}+Q^\delta_{ik}\omega^\delta_{kj}-\omega_{ik}^\delta Q^\delta_{kj}\non\\
&\quad =\zeta\Delta{Q^\delta_{ij}}  +2L_4 (  Q^\delta_{ij,l}Q^\delta_{l k})_{,k} -L_4 Q^\delta_{kl,i} Q^\delta_{kl,j} +
\frac{L_4}{2}|\nabla{Q^\delta}|^2 \delta_{ij}-aQ^\delta_{ij}-c \tr((Q^\delta)^2)Q^\delta_{ij}
\label{eq:1.18}
  \end{align}
for a.e. $(x,t)\in\T\times(0, T)$.

Since the approximate solution $(\bu^\delta,\,Q^\delta)$ satisfies the energy identity \eqref{lemma-energy-law}, then by
Lemma \ref{lemma1.0} and a similar argument in Section 4.2, we obtain the following estimates
\begin{align*}
  & \|\bu^\delta\|_{L^\infty(0,T; L^2(\T))\cap L^2(0,T;H^1(\T))}+ \sqrt{\delta} \|\mathcal{D}^k \bu^{\delta }\|_{L^2(0,T;L^2(\T))}\leq C,\\
  & \|Q^{\delta}\|_{L^\infty(0,T;H^1(\T))\cap L^2(0,T;H^2(\T))}\leq C,
\end{align*}
where the constant $C$ depends on $\|\bu_0\|_{L^2}$, $\|Q_0\|_{H^1}$, $\eta_1$, $\nu$, $\T$, $L_i$, $a$, $c$,
but it is independent of the parameter $\delta$.
The above estimates and the Sobolev embedding theorem yield that
\begin{equation}
\|\bu^\delta\|_{L^4(0,T;L^4(\T))}+\|\nabla Q^\delta\|_{L^4(0,T;L^4(\T))} \leq C,\label{nqL4L4}
\end{equation}
which together with \eqref{eq:1.24} implies
\begin{align*}
&\|\sigma^a(Q^\delta)\|_{L^2(0,T;L^2(\T))}+\|\sigma^s(Q^\delta)\|_{L^2(0,T;L^2(\T))}\leq C,
\end{align*}
where $C$ is again independent of $\delta$. Then by comparison, we have
\begin{align*}
\|\p_t \bu^\delta \|_{L^2(0,T; (H^k(\T))')}+\|\p_t Q\|_{L^2(0,T;L^2(\T))}\leq C.
\end{align*}

These uniform bounds imply that, up to the extraction of a subsequence, the following convergence results as $\delta\to 0^+$:
\begin{align*}
\bu^\delta \rightarrow \bu \quad &\text{weakly star in}~ L^\infty(0,T;L^2(\T))\cap L^2(0,T;H^1(\T)),\\
\p_t \bu^\delta \rightarrow \p_t \bu \quad &\text{weakly in}~ L^2(0,T;(H^{k}_\sigma(\T))'),\\
Q^\delta \rightarrow Q \quad &\text{weakly star in}~ L^\infty(0,T;H^1(\T)\cap L^\infty(\T))\cap L^2(0,T;H^2(\T)),\\
\p_t Q^\delta \rightarrow \p_t Q \quad &\text{weakly in}~ L^2(0,T;L^2(\T)),\\
\sqrt{\delta}\left(\sqrt{\delta} \mathcal{D}^k \bu^\delta\right) \rightarrow 0\quad &\text{strongly in}~ L^2(0,T;L^2(\T)).
\end{align*}
Besides, using the well-known Aubin--Lions compactness lemma (see e.g., \cite{S87}), we obtain the following
strong convergence results (up to a subsequence)
\begin{align}
  \bu^\delta &\rightarrow \bu~\text{strongly in}~ L^2(0,T;H_\sigma^{1-\epsilon}(\T)),\label{eq:1.23}\\
  Q^\delta &\rightarrow Q~\text{strongly in}~ L^2(0,T;H^{2-\epsilon}(\T))\cap L^4(0,T; L^4(\T)),
  \label{eq:1.23a}
\end{align}
for any $\epsilon\in (0,\frac12)$.
Concerning the convergence of nonlinear terms, we only
need to treat the cubic term associated with $L_4$ in $\sigma^s$ (see \eqref{tensors}).
It follows from \eqref{eq:1.23a} that $Q$ and $\nabla Q$ converge almost everywhere in $\T\times(0,T)$ (up to a subsequence).
Then we infer from \eqref{eq:1.24} and \eqref{nqL4L4} that
\begin{equation}
  \lim_{\delta\to 0^+}\int_0^T\int_{\T} (Q^\delta)_{jm}(Q^\delta)_{kl,m}(Q^\delta)_{kl,i} M_{ij} dx dt
  =\int_0^T\int_{\T} Q_{jm}Q_{kl,m}Q_{kl,i} M_{ij} dx dt,\non
\end{equation}
for any $M\in L^2(0,T; L^2(\T))$. The convergence of other nonlinear terms can be treated by
using a similar argument like in \cite[Section 3]{ADL14} and the details are omitted here.

In conclusion, we are able to pass to the limit $\delta\to 0^+$ in \eqref{eq:1.18} and deduce that the limit function
$Q$ satisfies \eqref{eq:1.19} almost everywhere in $\T\times(0, T)$.
Moreover, passing to the limit in \eqref{eq:1.22}, we get
\begin{align*}
 &\int_0^T\langle \p_t \bu, \mathbf{v}\rangle_{(H^{k}_\sigma)',H^k_\sigma} dt
 - \int_0^T\int_{\T} \bu\otimes \bu:\nabla \mathbf{v}dxdt +\nu\int_0^T\int_{\T}  \nabla{\bu}:
\nabla{\mathbf{v}} dxdt\non\\
&\quad = -\int_0^T\int_{\T} (\sigma^a(Q)+\sigma^s(Q)):\nabla \mathbf{v}dxdt,
\end{align*}
for any $\mathbf{v}\in L^2(0,T; H^k_{\sigma}(\T))$.
By comparison, the above identity also implies that the time derivative of $\bu$ fulfills
$\p_t \bu \in L^2(0,T; (H_\sigma^{1}(\T))')$
and the weak formulation \eqref{eq:1.21} follows.
Besides, we infer from a suitable interpolation inequality (recall \eqref{Gel}) that
$\bu \in C([0,T]; L^2_\sigma(\T))$ and $Q\in C([0,T]; H^1(\T))$.

Finally, we easily deduce from \eqref{Linfty} and the lower
semicontinuity property of the $L^\infty$-norm (see e.g., \cite[Proposition 3.13]{brezis2010functional}) that
the limit function $Q$ satisfies the uniform estimate \eqref{eq:1.20b} associated with its $L^\infty$-norm.
Furthermore, thanks to Proposition \ref{proposition on energy law},
the energy identity \eqref{lemma-energy-law} (now with $\delta=0$) is satisfied.
\end{proof}

%%%%%%%%%%%%%%%%%%%%%%%%%%%%%%%%%%%%%%%%%%%%%%%%%%%%%%%%%%%%%%%%%%%%%%%%%%%%%%%%%%%%%%%%%%%%%%
\subsection{Continuous dependence with respect to initial data}

Next, we aim to show the uniqueness of global weak solutions.
For this purpose, we prove a continuous dependence result with respect to
the initial data in a suitable topology.

A conventional approach is to estimate the difference between two solutions by
energy estimates that are usually performed at the natural energy space level,
namely, $(\bu, Q)\in L^\infty(0,T;L^2_\sigma(\T))\times L^\infty(0,T;H^1(\T))$ for our
current problem \eqref{NSE}--\eqref{IC}.
However, due to the highly nonlinear stress tensors $\sigma^a,\, \sigma^s$,
this goal cannot be achieved in such settings.
This was also illustrated in \cite{PZ12} where only a weak-strong uniqueness result was obtained in
the simpler isotropic case $L_2=L_3=L_4=0$.

Inspired by \cite{LLT13, LTX16}, we shall perform the energy
estimates at a lower-order energy space than that appeared in the
basic energy law \eqref{eq:1.25}. Roughly speaking, we use the $(H^1)'$
energy estimate for the velocity $\bu$ and $L^2$ energy estimate for
the order parameter $Q$, respectively. As we shall see below,
one advantage in such settings is that certain higher-order a
priori estimates (compared with $(H^1)'\times L^2$) are
automatically provided by the basic energy law.

\begin{lemma}\label{prop on continuous dependence}
Let $(\bu_i, Q_i)$, $i=1,2$, be two global weak solutions to
problem \eqref{NSE}--\eqref{IC} with initial data $(\bu_{0i}, Q_{0i})\in L^2_\sigma(\T)\times (H^1(\T)\cap L^\infty(\T))$. Let
\begin{align}
\eta_2=\min\left\{\frac{\sqrt{\nu}}{16C^*}\frac{\zeta}{|L_4|},\ \frac{1}{64}\left(\frac{\zeta}{L_4}\right)^2\right\},\label{eta2}
\end{align}
where $C^*$ is the constant in the elliptic estimate \eqref{H2} depending only on $\T$.
If
\begin{align}
\|Q_i\|_{L^\infty(0,T; L^\infty(\T))}\leq \sqrt{\eta_2},\label{smallness}
\end{align}
then we have
\begin{align}
&\|\bar{\mathbf{w}}(t)\|_{H^1}^2+\|\bar{Q}(t)\|_{L^2}^2+\int_0^t\big(\|\Delta \bar{\mathbf{w}}(s)\|_{L^2}^2+\|\nabla \bar{Q}(s)\|_{L^2}^2\big)ds\non\\
&\quad \leq Ce^{Ct}\big(\|\bar{\mathbf{w}}_0\|_{H^1}+\|\bar{Q}_0\|_{L^2} \big),\quad \forall\, t\in (0, T).
\label{inequality-continuous-dependence}
\end{align}
Here,
\begin{align*}
&\bar{\mathbf{w}}\defeq(-\Delta+I)^{-1}(\bu_1-\bu_2),\quad\quad \, \bar{Q}\defeq Q_1-Q_2,\\
&\bar{\mathbf{w}}_0\defeq(-\Delta+I)^{-1}(\bu_{01}-\bu_{02}),\quad \bar{Q}_0\defeq Q_{01}-Q_{02},
\end{align*}
$I$ stands for the identity operator and $C>0$ is a constant that depends on
$\|\bu_{0i}\|_{L^2}$, $\|Q_{0i}\|_{H^1}$ for $1\leq i\leq 2$, $\eta_2$ and coefficients of the system.
\end{lemma}
\begin{proof}
Let
\begin{align*}
&\bar{\bu}\defeq \bu_1-\bu_2, \quad \bar{P}\defeq P_1-P_2,\quad \bar{\sigma}^a\defeq \sigma^a(Q_1)-\sigma^a(Q_2),
\quad\bar{\sigma}^s\defeq \sigma^s(Q_1)-\sigma^s(Q_2),\\
&\bar\HH\defeq (\HH(Q_1)+\lambda_1 \mathbb{I}+\mu_1-\mu_1^T) -(\HH(Q_2)+\lambda_2 \mathbb{I}+\mu_2-\mu_2^T).
\end{align*}
First, we infer from the incompressibility condition \eqref{incompressibility} that
\begin{equation}\label{incompressibility-2}
\nabla\cdot\bar{\mathbf{w}}=\nabla\cdot[(-\Delta+I)^{-1}\bar{\bu}]=(-\Delta+I)^{-1}(\nabla\cdot\bar{\bu})=0.
\end{equation}
Besides, we see that $\bar{\mathbf{w}}$ satisfies the following equation
\begin{align}
\partial_t\bar{\mathbf{w}}&=(-\Delta+I)^{-1}\partial_t\bar{\bu}\non\\
&=(-\Delta+I)^{-1}\big[\bu_2\cdot\nabla{\bu}_2-\bu_1\cdot\nabla{\bu}_1+\nu \Delta\bar{\bu}-\nabla\bar{P}
+\nabla\cdot(\bar{\sigma}^a+\bar{\sigma}^s)\big].
\label{xi-equ}
\end{align}
Multiplying equation \eqref{xi-equ} with $\bar{\mathbf{w}}-\Delta\bar{\mathbf{w}}$, integrating over
$\T$, using \eqref{incompressibility-2}, we obtain  after integration by parts that
\begin{align}
&\frac12\frac{d}{dt} \int_{\T}(|\bar{\mathbf{w}}|^2+|\nabla\bar{\mathbf{w}}|^2)dx\non\\
&\quad =\int_{\T} [(I-\Delta)\bar{\mathbf{w}}] \cdot (-\Delta +I)^{-1} \big[\bu_2\cdot\nabla{\bu}_2 -\bu_1\cdot\nabla{\bu}_1+\nu\Delta\bar{\bu}-\nabla\bar{P} +\nabla\cdot(\bar{\sigma}^a+\bar{\sigma}^s)\big]\,dx\non\\
&\quad =-\left\langle\bar{\mathbf{w}},\nabla\cdot(\bu_1\otimes \bu_1-\bu_2\otimes \bu_2)-\nu \Delta\bar{\bu}+\nabla\bar{P}-\nabla\cdot(\bar{\sigma}^a+\bar{\sigma}^s)\right\rangle_{H^1, (H^1)'}\non\\
&\quad =\nu \langle\bar{\mathbf{w}},\Delta\bar{\bu}\rangle_{H^1, (H^1)'}
+\int_{\T}\nabla\bar{\mathbf{w}} : (\bu_1\otimes \bu_1-\bu_2\otimes \bu_2-\bar{\sigma}^a-\bar{\sigma}^s)\,dx\non\\
&\quad =-\nu \int_{\T}(|\nabla\bar{\mathbf{w}}|^2+|\Delta\bar{\mathbf{w}}|^2)dx
+\underbrace{\int_{\T}\nabla\bar{\mathbf{w}}:(\bar{\bu}\otimes\bu_1+\bu_2\otimes\bar{\bu})\,dx}_{I_1}\non\\
&\qquad
-\underbrace{\int_{\T}\nabla\bar{\mathbf{w}}:\bar\sigma^s\,dx}_{I_2}
-\underbrace{\int_{\T}\nabla\bar{\mathbf{w}}:\bar\sigma^a\,dx}_{I_3}.
\label{xi-energy-inequlity}
\end{align}
Using the estimates for global weak solutions
\begin{align}
&\|\bu_i\|_{L^\infty(0,T;L^2(\T))\cap L^2(0,T;H^1(\T))}\leq C, \quad \|Q_i\|_{L^\infty(0,T; H^1(\T))\cap L^2(0,T;H^2(\T))}\leq C, \quad 1\leq i\leq 2
\label{integrability}
\end{align}
and the $L^\infty$-estimate \eqref{smallness}, we proceed to estimate the terms $I_1$ to $I_3$.
By the H\"older and Young inequalities, we have
\begin{align*}
I_1
&\leq\int_{\T}|\nabla\bar{\mathbf{w}}|(|\bu_1|+|\bu_2|)(|\bar{\mathbf{w}}|+|\Delta\bar{\mathbf{w}}|)\,dx\non\\
&\leq C(\|\bu_1\|_{L^4}+\|\bu_2\|_{L^4}) \big(\|\nabla \bar{\mathbf{w}}\|_{L^4}\|\Delta \bar{\mathbf{w}}\|_{L^2}+\|\nabla \bar{\mathbf{w}}\|_{L^2}\|\bar{\mathbf{w}}\|_{L^4})\non\\
&\leq\frac{\nu}{32}\|\Delta\bar{\mathbf{w}}\|_{L^2}^2
+C(\|\bu_1\|_{L^4}^2+\|\bu_2\|_{L^4}^2)\|\nabla\bar{\mathbf{w}}\|_{L^4}^2+C(\|\bu_1\|_{L^4}+\|\bu_2\|_{L^4})\|\bar{\mathbf{w}}\|_{H^1}^2 \non\\
&\leq\frac{\nu}{32}\|\Delta\bar{\mathbf{w}}\|_{L^2}^2+C(\|\bu_1\|_{L^4}+\|\bu_2\|_{L^4})\|\bar{\mathbf{w}}\|_{H^1}^2\non\\
&\quad +C(\|\bu_1\|_{L^2}\|{\bu}_1\|_{H^1}+\|\bu_2\|_{L^2}\|{\bu}_2\|_{H^1})\|\nabla\bar{\mathbf{w}}\|_{L^2}(\|\Delta\bar{\mathbf{w}}\|_{L^2}+\|\bar{\mathbf{w}}\|_{L^2})\non\\
&\leq\frac{\nu}{16}\|\Delta\bar{\mathbf{w}}\|_{L^2}^2
+C(1+\|{\bu}_1\|_{H^1}^2+\|{\bu}_2\|_{H^1}^2)(\|\bar{\mathbf{w}}\|_{L^2}^2+\|\nabla\bar{\mathbf{w}}\|_{L^2}^2).
\end{align*}
Next, for $I_2$, it holds
\begin{align*}
I_2
&\leq 2(|L_1|+|L_2|+|L_3|)\|\nabla\bar{\mathbf{w}}\|_{L^4}\|\nabla\bar{Q}\|_{L^2}(\|\nabla{Q}_1\|_{L^4}+\|\nabla{Q}_2\|_{L^4})\non\\
&\quad+2|L_4|\|\nabla\bar{\mathbf{w}}\|_{L^4}\|Q_2\|_{L^\infty}\|\nabla\bar{Q}\|_{L^2}(\|\nabla{Q}_1\|_{L^4}+\|\nabla{Q}_2\|_{L^4})\non\\
&\quad +2|L_4|\|\nabla\bar{\mathbf{w}}\|_{L^4}\|\bar{Q}\|_{L^4}\|\nabla{Q}_1\|_{L^4}^2\non\\
&\leq
C\|\nabla\bar{\mathbf{w}}\|_{L^2}^{\frac12}\big(\|\Delta\bar{\mathbf{w}}\|_{L^2}+\|\bar{\mathbf{w}}\|_{L^2}\big)^{\frac12}\|\nabla\bar{Q}\|_{L^2}
\big(\|Q_1\|_{L^\infty}^\frac12\|\Delta{Q}_1\|_{L^2}^\frac12
+\|Q_2\|_{L^\infty}^\frac12\|\Delta{Q}_2\|_{L^2}^\frac12\big)\non\\
&\quad+C\|\nabla\bar{\mathbf{w}}\|_{L^2}^{\frac12}\big(\|\Delta\bar{\mathbf{w}}\|_{L^2}+\|\bar{\mathbf{w}}\|_{L^2}\big)^{\frac12}
\big(\|\bar{Q}\|_{L^2}^\frac12\|\nabla\bar{Q}\|_{L^2}^\frac12+\|\bar{Q}\|_{L^2}\big)\|Q_1\|_{L^\infty}\|\Delta{Q}_1\|_{L^2}\non\\
&\leq\frac{\nu}{16}\|\Delta\bar{\mathbf{w}}\|_{L^2}^2+\frac{\zeta^2}{16}\|\nabla\bar{Q}\|_{L^2}^2\non\\
&\quad +C(1+\|\Delta{Q}_1\|_{L^2}^2+\|\Delta{Q}_2\|_{L^2}^2)(\|\bar{\mathbf{w}}\|_{L^2}^2+\|\nabla\bar{\mathbf{w}}\|_{L^2}^2+\|\bar{Q}\|_{L^2}^2).
\end{align*}
Observe that
\begin{align*}
I_3
&=-\int_{\T}\nabla\bar{\mathbf{w}}: [\bar{Q}(\HH(Q_1)+\lambda_1 \mathbb{I}+\mu_1-\mu_1^T)]\,dx-\int_{\T}\nabla\bar{\mathbf{w}}: (Q_2\bar\HH)\,dx\non\\
&\quad +\int_{\T}\nabla\bar{\mathbf{w}}: [(\HH(Q_1)+\lambda_1 \mathbb{I}+\mu_1-\mu_1^T)\bar{Q}]\,dx+\int_{\T}\nabla\bar{\mathbf{w}}: (\bar\HH Q_2)\,dx\non\\
&:= I_{3a}+I_{3b}+I_{3c}+I_{3d}
\end{align*}
with obvious notation.
Concerning $I_{3a}+I_{3c}$, we obtain from the H\"older and Young inequalities that
\begin{align*}
I_{3a}+I_{3c}&\leq
\zeta\|\nabla\bar{\mathbf{w}}\|_{L^4}\|\bar{Q}\|_{L^4}\|\Delta Q_1\|_{L^2}+ C|L_4|\|\nabla\bar{\mathbf{w}}\|_{L^4}\|\bar{Q}\|_{L^4}\|Q_1\|_{L^\infty}\|Q_1\|_{H^2}\\
&\quad+C|L_4|\|\nabla\bar{\mathbf{w}}\|_{L^4}\|\bar{Q}\|_{L^4}\|\nabla{Q}_1\|_{L^4}^2+C\|\nabla\bar{\mathbf{w}}\|_{L^4}\|\bar{Q}\|_{L^4}(\|Q_1\|_{L^2}+\|Q_1\|_{L^6}^3)\\
&\leq
C\|\nabla\bar{\mathbf{w}}\|_{L^2}^{\frac12}\big(\|\Delta\bar{\mathbf{w}}\|_{L^2}+\|\bar{\mathbf{w}}\|_{L^2}\big)^{\frac12}
\big(\|\bar{Q}\|_{L^2}^{\frac12}\|\nabla\bar{Q}\|_{L^2}^{\frac12}+\|\bar{Q}\|_{L^2}\big) (1+\|\Delta{Q}_1\|_{L^2})\\
&\quad +C\|\nabla\bar{\mathbf{w}}\|_{L^2}^{\frac12}\big(\|\Delta\bar{\mathbf{w}}\|_{L^2}+\|\bar{\mathbf{w}}\|_{L^2}\big)^{\frac12}
\big(\|\bar{Q}\|_{L^2}^{\frac12}\|\nabla\bar{Q}\|_{L^2}^{\frac12}+\|\bar{Q}\|_{L^2}\big)\|Q_1\|_{L^\infty}\|\Delta{Q}_1\|_{L^2}\\
&\leq\frac{\nu}{16}\|\Delta\bar{\mathbf{w}}\|_{L^2}^2+\frac{\zeta^2}{16}\|\nabla\bar{Q}\|_{L^2}^2
+C(1+\|\Delta{Q}_1\|_{L^2}^2)(\|\bar{\mathbf{w}}\|_{L^2}^2+\|\nabla\bar{\mathbf{w}}\|_{L^2}^2+\|\bar{Q}\|_{L^2}^2).
\end{align*}
Next, we rewrite $I_{3b}+I_{3d}$ as
\begin{align}
I_{3b}+I_{3d}= \zeta \int_{\T}\nabla\bar{\mathbf{w}}: ( \Delta \bar Q Q_2-Q_2 \Delta \bar Q) dx + I_{3e}.\non
\end{align}
Then using H\"older and Young inequalities, we obtain that
\begin{align*}
I_{3e}&\leq  2|L_4|\|Q_1\|_{L^\infty}\big(\|\bar{\mathbf{w}}\|_{H^2}\|Q_2\|_{L^\infty}+\|\nabla\bar{\mathbf{w}}\|_{L^4}\|\nabla{Q}_2\|_{L^4}\big)
\|\nabla\bar{Q}\|_{L^2}\\
&\quad +2|L_4|\big(\|\bar{\mathbf{w}}\|_{H^2}\|Q_2\|_{L^\infty}+\|\nabla\bar{\mathbf{w}}\|_{L^4}\|\nabla{Q}_2\|_{L^4}\big)\|\nabla{Q}_2\|_{L^4}\|\bar{Q}\|_{L^4}\\
&\quad +C|L_4|\|\nabla\bar{\mathbf{w}}\|_{L^4}\|Q_2\|_{L^\infty}\|\nabla\bar{Q}\|_{L^2}(\|\nabla{Q}_1\|_{L^4}+\|\nabla{Q}_2\|_{L^4})\\
&\quad +C\|\nabla\bar{\mathbf{w}}\|_{L^2}\|Q_2\|_{L^\infty}(1+\|Q_1\|_{L^\infty}^2+\|Q_2\|_{L^\infty}^2)\|\bar Q\|_{L^2} \\
&\leq
2C^*|L_4|\|Q_1\|_{L^\infty}\|Q_2\|_{L^\infty}(\|\Delta\bar{\mathbf{w}}\|_{L^2} + \|\bar{\mathbf{w}}\|_{L^2})\|\nabla\bar{Q}\|_{L^2}\\
&\quad +C\|\nabla\bar{\mathbf{w}}\|_{L^2}^{\frac12}\big(\|\Delta\bar{\mathbf{w}}\|_{L^2}+\|\bar{\mathbf{w}}\|_{L^2}\big)^{\frac12}\|\nabla{Q}_2\|_{L^4}\|\nabla\bar{Q}\|_{L^2}\non\\
&\quad+C\big(\|\Delta\bar{\mathbf{w}}\|_{L^2} + \|\bar{\mathbf{w}}\|_{L^2}\big)\|\nabla{Q}_2\|_{L^4}
\big(\|\bar{Q}\|_{L^2}^{\frac12}\|\nabla\bar{Q}\|_{L^2}^{\frac12}+\|\bar{Q}\|_{L^2}\big)\\
&\quad+C\|\nabla\bar{\mathbf{w}}\|_{L^2}^{\frac12}\big(\|\Delta\bar{\mathbf{w}}\|_{L^2}+\|\bar{\mathbf{w}}\|_{L^2}\big)^{\frac12}
\|\nabla{Q}_2\|_{L^4}^2\big(\|\bar{Q}\|_{L^2}^{\frac12}\|\nabla\bar{Q}\|_{L^2}^{\frac12}+\|\bar{Q}\|_{L^2}\big)\\
&\quad+C\|\nabla\bar{\mathbf{w}}\|_{L^2}^{\frac12}\big(\|\Delta\bar{\mathbf{w}}\|_{L^2}+\|\bar{\mathbf{w}}\|_{L^2}\big)^{\frac12}
\|\nabla\bar{Q}\|(\|\nabla{Q}_1\|_{L^4}+\|\nabla{Q}_2\|_{L^4})\\
&\quad +C\|\nabla\bar{\mathbf{w}}\|_{L^2}\|\bar Q\|_{L^2},
\end{align*}
where $C^*$ is a constant depending only on $\T$ (see \eqref{H2}). By our choice of $\eta_2$ and \eqref{smallness}, we have
$$2C^*|L_4|\|Q_1\|_{L^\infty}\|Q_2\|_{L^\infty}\leq 2C^*|L_4|\eta_2^2=\frac{\sqrt{\nu}\zeta}{8}.$$
Then by Young's inequality, we obtain that
\begin{align*}
I_{3e}&\leq\frac{\nu}{16}\|\Delta\bar{\mathbf{w}}\|_{L^2}^2+\frac{\zeta^2}{4}\|\nabla\bar{Q}\|_{L^2}^2
+C(1+\|\Delta{Q}_1\|_{L^2}^2+\|\Delta{Q}_2\|_{L^2}^2)(\|\bar{\mathbf{w}}\|_{L^2}^2+\|\nabla\bar{\mathbf{w}}\|_{L^2}^2+\|\bar{Q}\|_{L^2}^2).
\end{align*}
Denote $$\bar \xi = \frac12(\nabla\bar{\mathbf{w}}-\nabla^T\bar{\mathbf{w}}).$$
We infer from Lemma \ref{cancel} that
\begin{align}
&\zeta \int_{\T}\nabla\bar{\mathbf{w}}: ( \Delta \bar Q Q_2-Q_2 \Delta \bar Q) dx\non\\
&\quad = \zeta\int_{\T} \nabla\bar{\mathbf{w}}: ( \Delta \bar Q Q_2-Q_2 \Delta \bar Q) dx \non\\
&\quad = -\zeta\int_{\T} (Q_2\bar\xi - \bar\xi Q_2): \Delta \bar Q dx\non\\
&\quad = -\zeta\int_{\T} (Q_2\Delta \bar\xi - \Delta \bar\xi Q_2) :  \bar Q dx - 2\zeta\int_{\T} (\p_i Q_2\p_i \bar\xi - \p_i \bar \xi\p_i Q_2) :  \bar Q dx\non\\
&\qquad -\zeta\int_{\T} (\Delta Q_2  \bar\xi -  \bar\xi \Delta Q_2) : \bar Q dx\non\\
&\quad := -\zeta\int_{\T} (Q_2\Delta \bar\xi - \Delta \bar\xi Q_2) :  \bar Q dx + I_{4a}+I_{4b},\non
\end{align}
with obvious notation. Using H\"older and Young inequalities, we obtain
\begin{align}
I_{4a}+I_{4b}
&\leq C\|\nabla Q_2\|_{L^4}\|\nabla \bar\xi\|_{L^4}\|\bar Q\|_{L^2} +C\|\Delta Q_2\|_{L^2}\|\bar \xi\|_{L^4}\|\bar Q\|_{L^4}\non\\
&\leq C\|Q_2\|_{L^\infty}^\frac12\|\Delta Q_2\|_{L^2}^\frac12 \|\nabla\bar{\mathbf{w}}\|_{L^2}^{\frac12}\big(\|\Delta\bar{\mathbf{w}}\|_{L^2}+\|\bar{\mathbf{w}}\|_{L^2}\big)^{\frac12}
\|\bar Q\|_{L^2}\non\\
&\quad +C\|\Delta Q_2\|_{L^2}(\| \bar{\mathbf{w}}\|_{L^2}^\frac12\|\nabla \bar{\mathbf{w}}\|_{L^2}^\frac12
+\|\bar{\mathbf{w}}\|_{L^2})(\|\bar Q\|_{L^2}^\frac12\|\nabla \bar Q\|_{L^2}^\frac12+\|\bar Q\|_{L^2})\nonumber\\
&\leq \frac{\nu}{4} \|\Delta \bar{\mathbf{w}}\|_{L^2}^2+ \frac{\zeta^2}{8}\|\nabla \bar Q\|_{L^2}^2
+C(1+ \|\Delta Q_2\|_{L^2}^2)   (\|\bar{\mathbf{w}}\|_{L^2}^2+\|\nabla \bar{\mathbf{w}}\|_{L^2}^2+\|\bar Q\|_{L^2}^2).\non
\end{align}
Collecting the above estimates, we conclude that
\begin{align}\label{uniqueness-estimate-1}
\frac{d}{dt}&\|\bar{\mathbf{w}}\|_{H^1}^2+2\nu\|\nabla\bar{\mathbf{w}}\|_{L^2}^2+\nu\|\Delta\bar{\mathbf{w}}\|_{L^2}^2\non\\
&\leq C(1+\|\bu_1\|_{H^1}^2+\|\bu_2\|_{H^2}^2+\|\Delta{Q}_1\|_{L^2}^2+\|\Delta{Q}_2\|_{L^2}^2)(\|\bar{\mathbf{w}}\|_{H^1}^2+\|\bar{Q}\|_{L^2}^2)\non\\
&\quad + \zeta^2 \|\nabla\bar{Q}\|^2 -2\zeta\int_{\T} (Q_2\Delta \bar\xi - \Delta \bar\xi Q_2) :  \bar Q dx.
\end{align}
Now we investigate the equation for $\bar{Q}$:
\begin{align}\label{equ-Q-bar}
&\partial_t\bar{Q}_{ij} +\bar{u}_k\partial_k({Q}_1)_{ij} +(u_2)_{k}\partial_k\bar{Q}_{ij}
+(\bar Q_{ik}({\omega}_1)_{kj}+({Q}_2)_{ik}\bar\omega_{kj})
-((\omega_1)_{ik}\bar{Q}_{kj}+\bar\omega_{ik}({Q}_2)_{kj})\non\\
&\quad =\zeta\Delta\bar{Q}_{ij}-a\bar{Q}_{ij}-c\tr^2(Q_1)\bar{Q}_{ij}-c[\bar{Q}:(Q_1+Q_2)](Q_2)_{ij}\non\\
&\qquad +2L_4\partial_k(\bar{Q}_{lk}\partial_l (Q_1)_{ij}+(Q_2)_{lk}\partial_l\bar{Q}_{ij})
        -L_4(\partial_i\bar{Q}_{kl}\partial_j(Q_1)_{kl}+\partial_i(Q_2)_{kl}\partial_j\bar{Q}_{kl})\non\\
&\qquad +\frac{L_4}{2}\partial_m\bar{Q}_{kl}\partial_m((Q_1)_{kl}+(Q_2)_{kl})\delta_{ij}.
\end{align}
Testing \eqref{equ-Q-bar} with $2\zeta\bar{Q}_{ij}$, summing over $i, j$ from $1$ to $2$, and then integrate over $\T$, by the incompressibility condition
\eqref{incompressibility} and the fact $(\omega_1\bar{Q}-\bar{Q}\omega_1):\bar{Q}=0$ (recall Lemma \ref{cancel}), we deduce
after integration by parts that
\begin{align}
&\zeta\frac{d}{dt}\|\bar{Q}\|_{L^2}^2 + 2\zeta^2\|\nabla\bar{Q}\|^2 \non\\
&\quad =-2\zeta\int_{\T}(\bar{\bu}\cdot\nabla{Q}_1):\bar{Q}\,dx
+2\zeta\int_{\T}(\bar\omega{Q}_2-Q_2\bar\omega):\bar{Q}\,dx \non\\
&\qquad -2\zeta\int_{\T}a|\bar{Q}|^2dx -2\zeta c\int_{\T}\left(\tr^2(Q_1)|\bar{Q}|^2+[\bar{Q}:(Q_1+Q_2)](Q_2:\bar{Q})\right)\,dx\non\\
&\qquad -4\zeta L_4\int_{\T}(Q_2)_{lk}\partial_l\bar{Q}_{ij}\partial_k\bar{Q}_{ij}\,dx
-4\zeta L_4\int_{\T}\bar{Q}_{lk}\partial_l(Q_1)_{ij}\partial_k\bar{Q}_{ij}\,dx\non\\
&\qquad -2\zeta L_4\int_{\T}(\partial_i\bar{Q}_{kl}\partial_j(Q_1)_{kl}+\partial_i(Q_2)_{kl}\partial_j\bar{Q}_{kl})\bar{Q}_{ij}\,dx\non\\
&\quad := \sum_{i=1}^7J_i.\label{dffQ}
\end{align}
Let us estimate the terms $J_1$ to $J_7$ individually. First, for $J_1$, we have
\begin{align*}
J_1&\leq \|\bar{\bu}\|_{L^2}\|\nabla{Q}_1\|_{L^4}\|\bar{Q}\|_{L^4}\\
&\leq
(\|\bar{\mathbf{w}}\|_{L^2}+\|\Delta\bar{\mathbf{w}}\|_{L^2})\|Q_1\|_{L^\infty}^\frac12\|\Delta{Q}_1\|_{L^2}^\frac12
\big(\|\bar{Q}\|_{L^2}^{\frac12}\|\nabla\bar{Q}\|_{L^2}^{\frac12}+\|\bar{Q}\|_{L^2}\big)\\
&\leq\frac{\zeta^2}{8}\|\nabla\bar{Q}\|_{L^2}^2 +\frac{\nu}{2}\|\Delta\bar{\mathbf{w}}\|_{L^2}^2
+C(1+\|\Delta{Q}_1\|_{L^2}^2)(\|\bar{\mathbf{w}}\|_{L^2}^2+\|\bar{Q}\|_{L^2}^2).
\end{align*}
Concerning $J_2$, we see from the identity $(-\Delta +I)^{-1} \bar \omega = \bar \xi$ that
\begin{align}
J_2
&=2\zeta\int_{\T}\big\{[(-\Delta +I)\bar\xi] {Q}_2- Q_2 [(-\Delta +I)\bar\xi]\big\} :\bar{Q}\,dx\non\\
&=2\zeta\int_{\T}\big(Q_2 \Delta \bar\xi- \Delta\bar\xi {Q}_2 \big) :\bar{Q}\,dx + 2\zeta\int_{\T}\big( \bar\xi {Q}_2- Q_2 \bar\xi\big) :\bar{Q}\,dx\non\\
&\leq 2\zeta\int_{\T}\big(Q_2 \Delta \bar\xi- \Delta\bar\xi {Q}_2 \big) :\bar{Q}\,dx +C \|Q_2\|_{L^\infty}\|\nabla \bar{\mathbf{w}}\|_{L^2}\|\bar{Q}\|_{L^2}\non\\
&\leq 2\zeta\int_{\T}\big(Q_2 \Delta \bar\xi- \Delta\bar\xi {Q}_2 \big) :\bar{Q}\,dx+ C(\|\nabla \bar{\mathbf{w}}\|_{L^2}^2+\|\bar{Q}\|_{L^2}^2).\non
\end{align}
By \eqref{smallness}, it easily follows  that
\begin{equation*}
J_3+J_4\leq C\|\bar{Q}\|^2,
\end{equation*}
and
\begin{align*}
J_5&\leq 4\zeta |L_4|\|Q_2\|_{L^\infty}\|\nabla\bar{Q}\|_{L^2}^2\leq\frac{\zeta^2}{2}\|\nabla\bar{Q}\|_{L^2}^2.
\end{align*}
Finally, we deduce that
\begin{align*}
J_6+J_7
&\leq C \|\nabla\bar{Q}\|_{L^2}\|\bar{Q}\|_{L^4}(\|\nabla{Q}_1\|_{L^4}+\|\nabla{Q}_2\|_{L^4})\\
&\leq C\|\nabla\bar{Q}\|_{L^2}\big(\|\bar{Q}\|_{L^2}^{\frac12}\|\nabla\bar{Q}\|_{L^2}^{\frac12}+\|\bar{Q}\|_{L^2}\big)
(\|Q_1\|_{L^\infty}^\frac12\|\Delta{Q}_1\|_{L^2}^\frac12+\|Q_2\|_{L^\infty}^\frac12\|\Delta{Q}_2\|_{L^2}^\frac12)\\
&\leq\frac{\zeta^2}{8}\|\nabla\bar{Q}\|_{L^2}^2+C(1+\|\Delta{Q}_1\|_{L^2}^2+\|\Delta{Q}_2\|_{L^2}^2)\|\bar{Q}\|_{L^2}^2.
\end{align*}
Summing up the estimates for $J_1$, ..., $J_7$, we conclude from \eqref{dffQ} that
\begin{align}\label{uniqueness-estimate-2}
&\zeta\frac{d}{dt}\|\bar{Q}\|_{L^2}^2 + \frac{3\zeta^2}{2}\|\nabla\bar{Q}\|_{L^2}^2\non\\
&\quad \leq \frac{\nu}{2}\|\Delta\bar{\mathbf{w}}\|_{L^2}^2+2\zeta\int_{\T}\big(Q_2 \Delta \bar\xi- \Delta\bar\xi {Q}_2 \big) :\bar{Q}\,dx\non\\
&\qquad +C(1+\|\Delta{Q}_1\|_{L^2}^2+\|\Delta{Q}_2\|_{L^2}^2)(\| \bar{\mathbf{w}}\|_{L^2}^2+\|\nabla \bar{\mathbf{w}}\|_{L^2}^2+\|\bar{Q}\|_{L^2}^2).
\end{align}
Finally, adding \eqref{uniqueness-estimate-1} with \eqref{uniqueness-estimate-2}, we obtain the following inequality
\begin{align}\label{diff-final}
&\frac{d}{dt}\big(\|\bar{\mathbf{w}}\|_{H^1}^2+\zeta\|\bar{Q}\|_{L^2}^2\big)
+ \frac{\nu}{2}\|\Delta\bar{\mathbf{w}}\|_{L^2}^2+\frac{\zeta^2}{2}\|\nabla\bar{Q}\|_{L^2}^2\non\\
&\quad \leq
C(1+\|\bu_1\|_{H^1}^2+\|\bu_2\|_{H^2}^2+\|\Delta{Q}_1\|_{L^2}^2+\|\Delta{Q}_2\|_{L^2}^2)(\|\bar{\mathbf{w}}\|_{H^1}^2+\|\bar{Q}\|_{L^2}^2).
\end{align}
Then applying Gronwall's inequality, we easily deduce the continuous dependence result \eqref{inequality-continuous-dependence} from \eqref{diff-final}.
\end{proof}
\medskip
As an immediate consequence of Lemma \ref{prop on continuous dependence}, we have
\begin{corollary}
Suppose the assumptions in Lemma \ref{prop on continuous dependence} are satisfied. The global
weak solution to problem \eqref{NSE}--\eqref{IC} is unique, if there exists any.
\end{corollary}
\begin{remark}
When the spatial dimension is two, for the special case $L_2=L_3=L_4=0$,
a weak-strong uniqueness result for the Cauchy problem of the Beris--Edwards system was given in \cite{PZ12} and later,
the uniqueness of weak solutions was proved in \cite{D15}. Our Lemma \ref{prop on continuous dependence} extends the previous results in the two dimensional periodic setting
and it still works in the case of whole space $\mathbb{R}^2$ (cf. \cite{LTX16}).
\end{remark}

\subsection{Proof of Theorem \ref{main-theorem}}
According to \eqref{eta1}, \eqref{eta2} and in view of \eqref{zk}, we are able
to choose the positive constants $K_1$, $K_2$ stated in Theorem
\ref{main-theorem}. In particular, we have
$0<\eta\leq\min\{\eta_1,\eta_2\}$, where $\eta_1$ and $\eta_2$ are
the constants in Proposition \ref{approx1} and Lemma \ref{prop on
continuous dependence}, respectively. Given any initial datum
$Q_0\in H^1(\T)\cap L^\infty(\T)$ with
$\|Q_0\|_{L^\infty}<\sqrt{\eta}$, there exits a sequence
$\{Q^n_0\}\subset H^2(\T)\hookrightarrow\hookrightarrow (H^1(\T)\cap L^\infty(\T))$ such that $Q^n_0\rightarrow Q_0$ strongly
in $H^1(\T) \cap L^\infty(\T)$. Furthermore, we can find $N\in
\mathbb{N}$ such that for $n\geq N$, it holds $\|Q_0^n\|_{L^\infty}<
\sqrt{\eta}$. Without loss of generality, we take $N=1$.

Since $\eta\leq \eta_1$, then for every pair of initial data $(\bu_0, Q_0^n)\in L^2_\sigma(\T)\cap H^2(\T)$,
we infer from Lemma \ref{prop on existence}
that problem \eqref{NSE}--\eqref{IC} admits a global weak solution  $(\bu_n, Q_n)$.
Thanks to $\eta\leq \eta_2$, we can apply Lemma  \ref{prop on continuous dependence} to conclude that
$\{(\bu_n, Q_n)\}$ is indeed a Cauchy sequence in $L^\infty(0, T; (H^1_\sigma(\T))')\times
L^\infty(0, T; L^2(\T))$, whose limit is denoted by $(\bu, Q)$.
Based on the dissipative energy law \eqref{eq:1.25} for $(\bu_n, Q_n)$
and the uniform bound on $\|Q_n\|_{L^\infty(0,T; L^\infty(\T))}$ (see \eqref{eq:1.20b}),
following a similar argument like in the proof of Lemma \ref{prop on existence},
it is standard to check that the limit function $(\bu, Q)$ is indeed a global weak solution
to problem \eqref{NSE}--\eqref{IC} with the initial data $(\bu_0, Q_0)\in L^2_\sigma(\T)\times (H^1(\T)\cap L^\infty(\T))$.
Finally, uniqueness of the global weak solution is a direct consequence of Lemma \ref{prop on continuous dependence}.

The proof of Theorem \ref{main-theorem} is complete.
\smallskip

%%%%%%%%%%%%%%%%%%%%%%%%%%%%%%%%%%%%%%%%%%%%%%%%%%%%%%%%%%%%%%%%%%%%%%%%%%%%%%%%%
\section{Appendix}
\setcounter{equation}{0}
For the convenience of the readers, we present the derivation of $\HH$  and $\HH+\lambda\Id+\mu-\mu^{T}$.

\begin{lemma}
The right-hand side of equation \eqref{Q equ} can be written as
\begin{align*}
 &(\HH+\lambda\Id+\mu-\mu^{T})_{ij}\\
&\quad =\zeta\Delta{Q_{ij}}+2L_4 (  Q_{ij,l}Q_{l k})_{,k} -L_4 Q_{kl,i} Q_{kl,j} + \frac{L_4}{2}|\nabla{Q}|^2 \delta_{ij} -aQ_{ij}-c \tr(Q^2)Q_{ij}
\end{align*}
and
\begin{align*}
  -\HH_{ij} &= -2L_1\Delta{Q_{ij}} -2(L_2+L_3) Q_{ik,kj}-2L_4 Q_{ij,\ell} Q_{l k,k}
-2L_4 Q_{ij,l k}Q_{l k}\non\\
&\quad +L_4 Q_{kl,i} Q_{kl,j}+aQ_{ij}-bQ_{jk}Q_{ki}+c\tr(Q^2)Q_{ij},
\end{align*}
for  $1\leq i,\,j,\,k,\,l\leq 2$.
\end{lemma}
\begin{proof}
We notice the $\HH$ is the minus variational derivative of $\E(Q)$ without imposing any constraints on $Q$.
A direct calculation yields that (see \cite[Proposition A.1]{IXZ14})
\begin{align*}
\left(\frac{\delta\mathcal{E}}{\delta Q}\right)_{ij}
&=-2L_1\Delta{Q_{ij}}-2(L_2+L_3)\partial_j\partial_kQ_{ik}-2L_4\partial_lQ_{ij}\partial_kQ_{lk}
-2L_4\partial_l\partial_kQ_{ij}Q_{lk}\\
&\qquad+L_4\partial_iQ_{kl}\partial_jQ_{kl}+aQ_{ij}-bQ_{jk}Q_{ki}+c\tr(Q^2)Q_{ij},
\end{align*}
which implies \eqref{Hdef}.

Substituting the above relation in~\eqref{Q equ} and choosing $\mu$ to enforce the
symmetry constraint $Q = Q^T$ yields that
$$
  \mu_{ij}-\mu_{ji}=(L_2+L_3)\left(\partial_i\partial_kQ_{jk}-\partial_j\partial_kQ_{ik}\right).
$$
Similarly, choosing $\lambda$ to enforce the trace free constraint $\tr(Q) = 0$ forces
$$
  \lambda=-\frac{b}{2}\tr(Q^2)-(L_2+L_3)\partial_l\partial_kQ_{lk}+\frac{L_4}{2}|\nabla{Q}|^2.
$$
Combining the expressions of $\lambda$, $\mu$ and $\mathcal{H}$ yields that
\begin{align}
  &(\HH+\lambda\Id+\mu-\mu^{T})_{ij}\non\\
&\quad =2L_1\Delta{Q_{ij}} +(L_2+L_3)(\partial_j\partial_kQ_{ik}+\partial_i\partial_kQ_{jk})+2L_4\partial_lQ_{ij}\partial_kQ_{lk}
\non\\
&\qquad +2L_4\partial_l\partial_kQ_{ij}Q_{lk}-L_4\partial_iQ_{kl}\partial_jQ_{kl}-aQ_{ij}+bQ_{jk}Q_{ki}-c
\tr(Q^2)Q_{ij}\non\\
&\qquad+\left(-\frac{b}{2}\tr(Q^2)-(L_2+L_3)\partial_l\partial_kQ_{lk}+\frac{L_4}{2}|\nabla{Q}|^2\right)\delta_{ij}.
\label{qequ}
\end{align}
Since $Q$ is a $2\times 2$ symmetric and traceless matrix, we can collect
the terms in \eqref{qequ} with factor $L_2+L_3$ and by a straightforward computation to show that
\begin{equation}
(L_2+L_3)(\partial_j\partial_kQ_{ik}+\partial_i\partial_kQ_{jk})-(L_2+L_3)\partial_l\partial_kQ_{lk}\delta_{ij}=(L_2+L_3)\Delta{Q}_{ij}.
\label{l2l3delta}
\end{equation}
 This together with \eqref{qequ} implies
\begin{equation*}
\begin{split}
&(\HH+\lambda\Id+\mu-\mu^{T})_{ij}\\
&\quad =(2L_1+L_2+L_3)\Delta{Q_{ij}}+2L_4\partial_lQ_{ij}\partial_kQ_{lk}
+2L_4\partial_l\partial_kQ_{ij}Q_{lk} -L_4\partial_iQ_{kl}\partial_jQ_{kl}\\
&\qquad -aQ_{ij}+bQ_{jk}Q_{ki}-c\tr(Q^2)Q_{ij} +\left(-\frac{b}{2}\tr(Q^2)+\frac{L_4}{2}|\nabla{Q}|^2\right)\delta_{ij}
\end{split}
\end{equation*}
and the desired result follows from \eqref{def of zeta} and the identity
$$
 b\Big(Q_{jk}Q_{ki}-\frac12  \tr(Q^2)\Big)\delta_{ij}=0.
$$
The proof is complete.
\end{proof}
\begin{remark}
(1) Since $\tr(Q^3)=0$ holds for any $Q\in \mathcal{S}_0^{(2)}$, the term $\frac{b}{3}\tr(Q^3)$ simply vanishes in $\E(Q)$. Thus, in the two dimensional system \eqref{NSE}--\eqref{Q equ} the term associated with $b$ does not appear explicitly (see the expression of $\HH+\lambda\Id+\mu-\mu^{T}$). Hence, we do not need to impose any condition on the coefficient $b$.

(2) The specific relation \eqref{l2l3delta} helps to derive the
above simplified form of $\HH+\lambda\Id+\mu-\mu^{T}$, which
however, only holds for $Q\in \mathcal{S}_0^{(2)}$. In the three
dimensional case, the terms associated with $L_1,\,L_2,\,L_3$ cannot
be rewritten as $(2L_1+L_2+L_3)\Delta Q$. Instead, they lead to an
anisotropic elliptic operator that satisfies the strong Legendre
condition (see e.g., \cite{HD14,LW16}).
Whether the result on existence of global weak solutions obtained Theorem \ref{main-theorem}
can be extended to the three dimensional case remains an open question.
\end{remark}

\bigskip

%%%%%%%%%%%%%%%%%%%%%%%%%%%%%%%%%%%%%%%%%%%%%%%%%%%%%%%%%%%%%%%%%%%%%%%%%%%%%%%%

\noindent \textbf{Acknowledgements.} Y.~N. Liu is supported by NNSFC
grant No. 11601334. H. Wu is supported by NNSFC grant No. 11631011 and the
Shanghai Center for Mathematical Sciences. 
Part of the work was commenced when X. Xu was visiting Shanghai Key Laboratory
for Contemporary Applied Mathematics and School of Mathematical
Sciences at Fudan University in May 2017 and May 2018, whose
hospitality is gratefully acknowledged.

%%%%%%%%%%%%%%%%%%%%%%%%%%%%%%%%%%%%%%%%%%%%%%%%%%%%%%%%%%%%%%%%%%%%%%%%%%%%%%%%%%%%%%%%%%%%%%%%%%%%%%%%%%%%%%%%%%%%%%%%%%

\end{document}